%% file: mainnew.tex
\documentclass[review,onefignum,onetabnum]{siamart190516}
\usepackage[all]{xy}
\usepackage{amsmath}
\usepackage{mathrsfs,amssymb,amsmath,booktabs,array,xcolor,url,color,soul,multirow,multicol}
\usepackage{mathbbold,bbm}
\usepackage{stfloats}
\usepackage{amsfonts}
\usepackage{extarrows}
\usepackage{textcomp}
\usepackage{cases}
\usepackage{bm}
\usepackage{arydshln}
\usepackage{graphicx}      
\usepackage{hyperref}
\usepackage{extarrows}
\usepackage{algorithm}
\usepackage{algorithmic}

\input{ex_shared}

\ifpdf
\hypersetup{
  pdftitle={Stability Analysis of some chemical reaction networks using Lyapunov Function PDEs},
  pdfauthor={Yafei~Lu, Chuanhou~Gao}
}
\fi

\def\dd{\text{d}}
\newtheorem{Property}{Property}
\newtheorem{example}{example}

\begin{document}

\maketitle
 \begin{abstract}
 This paper contributes to extending the validity of Lyapunov function PDEs (invented by Fang and Gao in [\textit{SIAM Journal on Applied Dynamical Systems}, 18(2019), pp. 1163-1199] and whose solution is conjectured to be able to behave as a Lyapunov function) in stability analysis to more mass-action chemical reaction networks. By defining a new class of networks, called complex balanced produced networks, we have proved that the Lyapunov function PDEs method is valid in capturing the asymptotic stability of this class of networks, and also to their compound with any 1-dimensional independent network according to species and with any two-species autocatalytic non-independent network if some moderate conditions are included. A notable point is that these three classes of networks are non-weakly reversible, any dimensional and of any deficiency. We apply our results to some practical biochemical reaction networks including birth-death processes, motifs related networks etc., to illustrate validity.

\end{abstract}

\begin{keywords}
chemical reaction networks, complex balanced produced networks, mass-action systems, Lyapunov function PDEs, asymptotic stability
\end{keywords}

\begin{AMS}
  34D20, 80A30, 93C15, 93D20
\end{AMS}

\section{Introduction}
It has been extensively recognized that chemical reaction networks (CRNs) appear in chemistry, biology and process industries. The study of CRNs, namely, CRN theory, originating from the well-known literature \cite{Feinberg1972Complex,HornJackson1972General}, aims at exploring the correlation between the dynamical properties and structural features of networks. In particular,
as the emergence of the discipline--systems biology, CRN theory has received considerable critical attention \cite{Sontag2012Structure} once again, as a powerful tool to analyze and explain the underlying dynamical behaviors of chemical and biochemical networks from the mathematical point of view. Starting from this practical point, we are encouraged to develop CRN theory, especially the research on dynamical characteristics, which may be stability, oscillation, persistence, etc, as to serve the real world better.
As results for guaranteeing stability of steady states of mass action system are few and far between (and well-earned, when they are discovered), the topic is of broad interest. In this paper, we mainly pour our attention to the stability analysis of equilibria, with the help of constructing suitable Lyapunov functions.

Regarding the stability property there have been plentiful findings \cite{Angeli2009A,Craciun2006Multiple,Feinberg1972Complex,Feinberg1995The,Finberg1987Chemical,Finberg1988Chemical,Szederke2011Finding} concentrated on certain CRNs with special structures, such as detailed balancing, complex balancing, among which the zero deficiency theorem \cite{HornJackson1972General} probably is the best known. It has been argued that following mass-action kinetics, a class of CRNs equipped with zero deficiency and weakly reversible structure is complex balanced regardless of specific parameter values, and possess only one equilibrium in every positive stoichiometric compatibility class. More importantly,  it has also pointed out that each of these equilibria is asymptotically stable by taking the pseudo-Helmholtz free energy function as the Lyapunov function. Later this work has been improved by \cite{Finberg1988Chemical}, known as one deficiency theorem, which reported the existence and uniqueness of equilibria for a complex balanced MAS within restrictions on deficiency (not necessarily zero).  Furthermore, based on these results, global asymptotic stability \cite{Siegel2000Global} of equilibria for complex balanced MASs can be achieved when gratifying persistent condition \cite{Angeli2007A,Craciun2013Persistence,Pantea2012On}. Besides, some works \cite{Necessary1989,rao2013graph} focused on the detailed balanced MAS with reversible structure, which is a special case of a complex balanced MAS. Particularly, van der Schaft and his coauthors \cite{rao2013graph} proposed a compact formulation to depict the network dynamics by using graph theory for this class of MASs, and the stability properties were also achieved.

For the sake of stability analysis of MASs, one widely accepted approach is to look for proper Lyapunov functions according to their structural properties.
Study of \cite{Johnston2011Linear,Johnston2012Dynamical} has shown that when the considered MAS can be mapped into a complex balanced MAS through linear conjugacy method, it resulted in the same stability for both systems. Alradhawi and Angeli \cite{Alradhawi2016New} established piecewise linear in rates Lyapunov functions for some balanced MASs, and further, they suggested the asymptotic stability property if LaSalle's condition was met. Another Lyapunov function candidate coming from \cite{Ke2019Complex}, called generalized pseudo-Helmholtz function, served to establish
asymptotic stability for a general balanced MAS which is defined on the notions of reconstructions and reverse reconstructions.

Distinguished with the above results, several attempts have been made to address the stability problem from the micro angle.
Anderson \cite{Anderson2015Lyapunov} put forward
the scaling limits of nonequilibrium potential as the Lyapunov function for birth-death systems. Fang and Gao \cite{Fang2015Lyapunov} developed partial differential equations (PDEs) from the chemical master equation for general MASs, termed as Lyapunov function PDEs, whose solutions are potential to become Lyapunov functions under required conditions. This systematic approach has been confirmed well on multiple CRNs, including complex balanced CRNs, CRNs of 1-dimensional stoichiometric
subspace and a few special cases of higher dimensional
stoichiometric subspace.

Nevertheless, those CRNs with more general structures remain an arduous obstacle. In the meanwhile, a bold conjecture was proposed in \cite{Fang2015Lyapunov}, it said that for any MAS admitting a stable positive equilibrium,
Lyapunov function PDEs can produce Lyapunov functions to render the system locally asymptotically stable at the equilibrium when choosing a proper boundary complex set.
Recent evidence \cite{Wu2020A} takes a small further step for this guess, which indicated the PDEs are also valid for a class of complex balanced produced (CBP) CRNs with non-weakly reversible structure.

Motivated by these exciting results, our paper extends the validity of Lyapunov function PDEs in asymptotic stability analysis to three kinds of networks with high-dimension, arbitrary deficiency as well as non-weakly reversible structure. Meanwhile, our work gives a great support for the conjecture again. Following the CBP CRNs proposed in \cite{Wu2020A}, an algorithm is proposed to demonstrate how to generate CBP CRNs from a complex balanced CRN. Further, we define a class of CRNs composed of a CBP CRN and some $1$-dimensional independent networks, referred to as CBP-$\ell$Sub1 CRNs. Inspired by the research on autocatalytic reactions \cite{hoessly2019stationary},  a type of meaningful CRNs which plays a vital role in the processes of life \cite{Kauffman1995At,Hordijk2004}, such as biological metabolism, the initial transcripts of rRNA, etc., we define a CRN compounded of a CBP CRN and any two-species autocatalytic CRNs, named as CBP-$\ell$ts-Autoca CRNs. For each kinds of the MASs mentioned above, we succeed in capturing their asymptotic stability properties by using Lyapunov function PDEs approach.  Moreover, it has been proved that a CBP MAS / CBP-$\ell$ts-Autoca MAS admits a unique/ at most one positive equilibrium in each positive stoichiometric compatibility class. Besides, a dimensionality reduction strategy has been proposed, which aims at obtain the stability property for a network by decomposing it into a CBP CRN and several $1$-dimensional CRNs.

The remainder of this paper proceeds as follows. \cref{sec2} reviews the relevant notations and conclusions concerning chemical reaction networks and Lyapunov function PDEs, as well as a crucial conjecture about PDEs. \cref{sec3} summarizes that each positive stoichiometric compatibility class induced by a CBP MAS contains a unique positive equilibrium and this equilibrium is asymptotically stable, and moreover, the global asymptotic stability can be reached if given persistent condition. Meanwhile, an algorithm is proposed to compute CBP MASs.
In \cref{sec4}, the number of positive equilibria in each positive stoichiometric compatibility class for a CBP-$\ell$Sub1 MAS is discussed at first. Then follows the asymptotic stability of this MAS by using Lyapunov function PDEs method, and two examples are well studied to illustrate these results. Furthermore, we put forward a dimensionality reduction method for decomposing a CRN into a $1$-dimensional MAS and a CBP MAS to achieve stability. Finally, \cref{sec6} concludes the paper.

\noindent{\textbf{Mathematical Notations:}}\\
\rule[1ex]{\columnwidth}{0.8pt}
\begin{description}
   \item[\hspace{-0.5em}{$\mathbb{R}^n, \mathbb{R}^n_{\geq 0}, \mathbb{R}^n_{>0}$}]:
	$n$-dimensional real space, non-negative real space, positive real space, respectively.
	\item[\hspace{-0.5em}{$\mathbb{Z}^n_{\geq 0}$}]: $n$-dimensional non-negative integer space.
	\item[\hspace{-0.5em}{$x^{v_{\cdot i}}$}]: $x^{v_{\cdot i}}=\prod_{j=1}^{d}x_{j}^{v_{ji}}$, where $x\in \mathbb{R}^{d}, v_{\cdot i}\in\mathbb{Z}^{d}$ and $0^{0}=1$.
	\item[\hspace{-0.5em}{$\frac{x}{y}$}]: $\frac{x}{y}=(\frac{x_1}{y_1}, \cdots, \frac{x_n}{y_n}) $, where $x\in\mathbb{R}^{n}$, $y\in\mathbb{R}^{n}_{>0}$.
	\item[\hspace{-0.5em}{$\mathrm{Ln}(x)$}]: $\mathrm{Ln}(x)=\left(\ln{x_{1}}, \cdots, \ln{{x}_{n}} \right)^{\top}$, where $x\in\mathbb{R}^{n}_{>{0}}$.
	\item[\hspace{-0.5em}{$\mathscr{C}^2(\cdot~ ; *)$}]:
	 the set of $i$th continuous differentiable functions from "$\cdot$" to "*".
	 \item[\hspace{-0.5em}{s.t.}]: such that.
\end{description}
\rule[1ex]{\columnwidth}{0.8pt}

\section{Preliminaries}\label{sec2}
In this section, we will provide a basic conceptual framework of CRNs and Lyapunov function PDEs for the understanding of subsequent results.

\subsection{Chemical reaction networks}
Consider a network involved with $n$ species $S_1, \cdots, S_n$ and $r$ chemical reactions. The $i$th $(i=1,\cdots,r)$ reaction is written as
\begin{equation*}
\sum^{n}_{j=1}v_{ji}S_j \rightarrow \sum^{n}_{j=1}v'_{ji}S_j,
\end{equation*}
where $v_{ji}, ~v'_{ji}\in\mathbb{Z}_{\geq 0}$ represent the complexes of reactant and resultant, respectively.
Following with \cite{Feinberg1995The}, here come some elementary definitions in regard to CRNs.
\begin{definition} \emph{(CRN)}. A CRN consists of three finite sets:
\begin{enumerate}
\item{a set of species $\mathcal{S}=\{S_1, \cdots, S_n\}$;}
\item{a set of complexes $\mathcal{C}=\bigcup^r_{i=1}\{v_{\cdot i},v'_{\cdot i}\}$ with $\text{Card}(\mathcal{C})=c$, and the $j$th entry of $v_{\cdot i}$ represents the stoichiometric coefficient of $S_j$ in this complex;}
\item{a set of reactions $\mathcal{R}=\{v_{\cdot 1}\rightarrow v'_{\cdot 1},\cdots,v_{\cdot r}\rightarrow v'_{\cdot r}\}$, which satisfies that $\forall~ v_{\cdot i} \in \mathcal{C}, v_{\cdot i}\rightarrow v_{\cdot i}\notin \mathcal{R}$ but $\exists~ v'_{\cdot i}$, s.t. $v_{\cdot i}\rightarrow v'_{\cdot i} \in\mathcal{R}$ or $v'_{\cdot i}\rightarrow v_{\cdot i}\in\mathcal{R}$.}
\end{enumerate}
The triple $(\mathcal{S,C,R})$ is often used to represent a CRN.
\end{definition}

\begin{definition}\emph{(stoichiometric subspace)}.
For a CRN $(\mathcal{S,C,R})$, the linear subspace $\mathscr{S}\triangleq \emph{span}\{v'_{\cdot 1}-v_{\cdot 1},\cdots,v'_{\cdot r}-v_{\cdot r}\}$ is called the stoichiometric subspace of this network, and $\emph{dim}\mathscr{S}$ represents the dimension of $\mathscr{S}$.
\end{definition}

\begin{definition}\emph{(stoichiometric compatibility class).}
 Let $\mathscr{S}$ be the stoichiometric subspace of a  CRN $(\mathcal{S,C,R})$ and $x_0\in\mathbb{R}_{\geq0}^{n}$, then the sets $\mathscr{S}(x_0)\triangleq\{x_0+\xi\mid\xi\in\mathscr{S}\}$, $\bar{\mathscr{S}}^+(x_0)\triangleq\mathscr{S}(x_0)\bigcap \mathbbold{R}^n_{\geq 0}$ and $\mathscr{S}^+(x_0) \triangleq\mathscr{S}(x_0)\bigcap \mathbbold{R}^n_{>0}$ are called the stoichiometric compatibility class, nonnegative and positive stoichiometric compatibility class of $x_0$, respectively.
\end{definition}

When a CRN takes with the mass-action kinetics, the reaction rate of the $i$th reaction $v_{\cdot i}\rightarrow v'_{\cdot i}$ is evaluated by
$$R_i(x)\triangleq k_{i} x^{v_{\cdot i}}=\prod_{j=1}^{d}x_{j}^{v_{ji}}$$
with $k_i \in \mathbb{R}_{>0}$, $x\in \mathbb{R}^n_{\geq 0}$ representing the rate constant of this reaction and the vector of concentrations $x_i$ of the substance $\mathcal{S}_i$.

\begin{definition}\emph{(mass-action system).}
A CRN $(\mathcal{S,C,R})$ assigned mass-action kinetics is said to be a mass-action system (MAS), often represented by the quadruple $\mathcal{M}\triangleq(\mathcal{S,C,R,K})$.
\end{definition}

The dynamics of an $\mathcal{M}$ that depicts the evolution of concentrations of the species over time is presented as
\begin{equation}\label{eq:mas}
\frac{\mathrm{d}x}{\mathrm{d}t}=\Gamma R(x),~~~x\in \mathbb{R}_{\geq0}^{n} ,
\end{equation}
where $\Gamma\in\mathbb{Z}_{n\times r}$ is the stoichiometric matrix with the $i$th column given by $\Gamma _{\cdot i}=v'_{\cdot i}-v_{\cdot i}$ called the reaction vector, and $R(x)$ is the vector function of reaction rate defined in $\mathbb{R}^{r}_{\geq 0}$ with each element $R_{i}(x)=k_{i}x^{v_{\cdot i}}$.

\begin{definition}\emph{(balanced MAS).}\label{equilibrium}
 A point $x^{*}\in\mathbb{R}_{>0}^{n}$ is said to be a positive equilibrium in $\mathcal{M}$ if it satisfies $\Gamma R(x^{*})=0$. A MAS that possesses a positive equilibrium is a balanced MAS.
\end{definition}

\begin{definition}\emph{(complex balanced MAS).}
For an $\mathcal{M}$, if $\exists x^*\in\mathbb{R}^n_{>0}$, s.t.
\begin{align}
\sum_{\{i\mid v_{\cdot i}=z\}}k_{i} (x^{*})^{v_{\cdot i}}=
\sum_{\{i\mid v'_{\cdot i}=z\}}k_{i}(x^{*})^{v_{\cdot i}}, \qquad \forall z\in \mathcal{C},
\end{align}
which means the consuming rate equals
the producing rate at this state for any complex,
then $x^*$ is called a complex balanced equilibrium, and this $\mathcal{M}$ is called a complex balanced MAS.
\end{definition}

\begin{definition}\emph{(reaction vector balanced MAS \cite{Cappelletti2018Graphically}).}
For an $\mathcal{M}$, if $\exists x^*\in\mathbb{R}^n_{>0}$, s.t.
\begin{align}
\sum_{\{i \mid v'_{\cdot i}-v_{\cdot i}=\eta\}}k_{i} (x^{*})^{v_{\cdot i}}=
\sum_{\{i\mid v'_{\cdot i}-v_{\cdot i}=-\eta\}}k_{i}(x^{*})^{v_{\cdot i}}, \qquad \forall \eta\in \mathbb{R}^n
\end{align}
then $x^*$ is a reaction vector balanced equilibrium in $\mathcal{M}$, and the MAS is a reaction vector balanced MAS.
\end{definition}

We use the following example to illustrate the reaction vector balanced equilibrium.

\begin{example}
A MAS takes the reaction route like
\begin{align*}
\xymatrix{S_2 \ar[r]^-{ k_1} & S_1,~ S_1+S_2\ar[r]^-{k_2}&2S_2,~
2S_1 \ar[r]^-{ k_3} & 2S_2,~
3S_2\ar[r]^-{ k_4}&2S_1+S_2.}
\end{align*}
There are four reaction vectors $(1,-1)^\top,~
(-1,1)^\top,~(-2,2)^\top,~(2,-2)^\top$ in the network. We only need to consider two of them, i.e., $\eta=(-1,1)^\top$ and $\eta=(-2,2)^\top$, respectively. If a positive concentration vector $x^*=(x_1^*,x_2^*)^{\top}$ is a reaction vector equilibrium, then for $\eta=(-1,1)^\top$ it should satisfy $k_2 x^*_1 x^*_2=k_1 x^*_2$, i.e., $x^*_1=\frac{k_1}{k_2}$, while for $\eta=(-2,2)^\top$ there should be $k_3 x^{*^2}_1=k_4x^{*^3}_2$, i.e., $x^*_2=\sqrt[3]{\frac{k_3}{k_4}\frac{k_1^2}{k_2^2}}$. \end{example}

\subsection{Lyapunov function PDEs}\label{sec2.2}For any balanced $\mathcal{M}$, Fang and Gao \cite{Fang2015Lyapunov} invented the Lyapunov function PDEs to the stability of $\mathcal{M}$, whose concrete form are
\begin{equation}\label{eq:LyapunovPDE}
\sum^{r}_{i=1}k_{i}x^{v_{\cdot i}}-\sum^{r}_{i=1}k_{i}x^{v_{\cdot i}}\exp\big\{ (v'_{\cdot i}-v_{\cdot i})^{\top}\nabla f(x)\big\}=0,
\end{equation}
where $x\in\mathbb{R}_{>0}^{n}$, accompanied with a boundary condition,
\begin{equation}\label{eq:boundary}
\lim_{x\rightarrow \bar{x}
	\atop
{x}\in (\bar{x}+\mathscr{S})\cap\mathbb{R}_{>0}^{n} }
\sum_{\{i\mid v_{\cdot i}\in\mathcal{C}_{\bar{x}}\}}k_{i}x^{v_{\cdot i}}-\sum_{\{i\mid v'_{\cdot i}\in\mathcal{C}_{\bar{x}}\}}k_{i}x^{v_{\cdot i}} \exp\big\{(v'_{\cdot i}-v_{\cdot i})^{\top}\nabla f(x)\big\}=0,
\end{equation}
where $\mathcal{C}_{\bar x}$ stands for the complex set induced by any boundary point $\bar{x}\in\partial\mathbb{R}_{\geq 0}^{n}$.
One simple alternative is the naive boundary complex set (see details in \cite{Fang2015Lyapunov}), defined as
\begin{align*}\label{eq:naivebd}
\bar{\mathcal{C}}_{\bar{x}}=\{z\in \mathcal{C} \mid \exists ~\epsilon >0, ~\text{such that}~\forall j=1,\cdots,n, \bar{x}_j\geq\epsilon z_j \}.
\end{align*}

The solutions of the Lyapunov function PDEs exhibit good properties in characterizing the dynamical behaviors of the corresponding MAS.

\begin{Property}\emph{(\cite{Fang2015Lyapunov})}
Given a balanced $\mathcal{M}$, assume there exists a solution $f(x)$ defined in $\mathscr{C}^1(\mathbb{R}^n_{>0};\mathbb{R})$ for the Lyapunov function PDE \cref{eq:LyapunovPDE} with proper boundary condition induced by $\mathcal{M}$, then $f(x)$ possesses the following two properties:
\begin{enumerate}
\item {$f(x)$ is dissipative, that is $\dot {f}(x)=\frac{\dd f(x)}{\dd t}\leq 0$ with the equality holding if and only if $\nabla f(x) \bot \mathscr{S}$;}
\item{if $f(x)$ is defined in $\mathscr{C}^2(\mathbb{R}^n_{>0};\mathbb{R})$, and $\exists\mathcal{D} \subset \mathbb{R}^n_{>0}$, s.t. $\forall x\in\mathcal{D}$ and $\forall \mu \in \mathscr{S}$, there is
\begin{align}\label{property2}
\mu^\top \nabla^2 f(x)\mu \geq 0,
\end{align}
where the equality holds if and only if $\mu=\mathbbold{0}_{n}$, then $\forall x\in \mathcal{D}$, $\dot {f}(x)=0$ if and only if  $x$ is an equilibrium in $\mathcal{M}$.}
\end{enumerate}
\end{Property}

A sufficient condition is then given to reach the asymptotic stability of MASs based on the solution of the Lyapunov function PDEs.

\begin{theorem}\emph{(\cite{Fang2015Lyapunov})}\label{asymstability}
Given an $\mathcal{M}$ with an equilibrium $x^* \in \mathbb{R}^{n}_{>0}$, assume that its Lyapunov function PDE \cref{eq:LyapunovPDE} admits a solution $f\in \mathscr{C}^2(\mathbb{R}^n_{>0};\mathbb{R})$, and there exits a region near $x^*$ such that \cref{property2} holds throughout this region. Then for any initial condition in this region but with an initial energy lower than that at any boundary point included in the region, the solution $f(x)$ is an available Lyapunov function to establish the locally asymptotic stability of $x^*$.
\end{theorem}

It is thus conjectured \cite{Fang2015Lyapunov} that, \emph{``for any MAS that admits a stable positive equilibrium, if the boundary complex set is equipped properly, then the Lyapunov function PDEs induced by this system have a solution qualified as a Lyapunov function to suggest the system is locally asymptotically stable at the equilibrium"}.

The conjecture has been proved true in the cases of the following three classes of MASs.

(1) complex balanced MASs whose PDEs admit the well-known pseudo-Helmholtz free energy function
\begin{align}\label{eq:Helmholtz}
G(x)=\sum^{n}_{j=1}\left(x^*_j-x_j-x_j\ln{\frac{x^*_j}{x_j}}\right),~~ x\in\mathbb{R}^n_{>0}
\end{align}
to be a solution that can act as a Lyapunov function.

(2) all MASs of $\text{dim}\mathscr{S}=1$ whose PDEs have a solution in the form of
\begin{align}\label{sub1Lya}
f(x)=\int^{\gamma(x)}_0 \ln {u}(y^{\dag}(x)+\alpha \omega)\dd \alpha
\end{align}
as an available Lyapunov function, where the definitions of $\gamma(x),y^{\dag}(x),u$, and $\omega$ are given in \cref{thm:sub1}.

(3) Com-$\ell$Sub1 MASs with $\text{dim}\mathscr{S}\geq 2$ composed of a complex balanced MAS (${\mathcal{S}}^{(0)},{\mathcal{C}}^{(0)}, {\mathcal{R}}^{(0)},{\mathcal{K}}^{(0)}$) and some $1$-dimensional MASs $({\mathcal{S}}^{(p)},{\mathcal{C}}^{(p)}, {\mathcal{R}}^{(p)},{\mathcal{K}}^{(p)})$, where $p=1,\cdots,\ell$ and all subnetworks are supposed to be mutually independent according to the species. The Lyapunov function PDEs admit a solution
$$F(x)=G(x^{(0)})+\sum^{\ell}_{p=1} f(x^{(p)}),$$
where $G(x^{(0)})$ and every $f(x^{(p)})$ are defined by \cref{eq:Helmholtz} and \cref{sub1Lya}, respectively. Clearly, $F(x)$ is a suitable Lyapunov function.

In the current work, we continue to exhibit the validity of the Lyapunov function PDEs to more CRNs. Based on the above known solutions, we try to construct more solutions as well as CRNs with special structures to validate the conjecture. For simplicity, we are only concerned with the Lyapunov function PDE \cref{eq:LyapunovPDE} regardless of its boundary condition in the subsequent study. It is naturally not difficult to prove that the solution of the PDE \cref{eq:LyapunovPDE} also supports the corresponding boundary condition.

%


\section{Stability of CBP MASs}\label{sec3}
In this section, we will define a new class of CRNs based on complex balanced ones, and demonstrate some nice results on stability for them using the Lyapunov function PDEs method. 


\subsection{Definition}
The concept of reverse reconstruction \cite{Ke2019Complex} for a MAS stimulates us to define a wide range of CRNs, which essentially originate from complex balanced MASs. We thus name them CBP CRNs \cite{Wu2020A}.

\begin{definition}\label{df:CBP-CRNs}\emph{(CBP MAS).}	
Given a complex balanced $\mathcal{M}$ governed by \cref{eq:mas} with an equilibrium $x^* \in \mathbb{R}^{n}_{>0}$, an $\tilde{\mathcal{M}}=(\tilde{\mathcal{S}},\tilde{\mathcal{C}},\tilde{\mathcal{R}},\tilde{\mathcal{K}})$ is called a CBP MAS with respect to $\mathcal{M}$ if for some positive diagonal matrix $D=\emph{diag}(d_1,~\cdots, ~d_n)$ but not the identity matrix, its species set admits $\tilde{\mathcal{S}}=\mathcal{S}$ while the complexes set $\tilde{\mathcal{C}}=\cup^{\tilde{r}}_{i=1}\{\tilde{v}_{\cdot i}, ~\tilde{v}'_{\cdot i}\}$
and the reactions set $\tilde{\mathcal{R}}=\cup^{\tilde{r}}_{i=1}\{\tilde{v}_{\cdot i}\stackrel{\tilde{k}_i}\longrightarrow\tilde{v}'_{\cdot i}\}$ satisfy

\emph{(1)} {$\tilde{r}=r,\tilde{v}_{\cdot i},\tilde{v}'_{\cdot i}\in\mathbb{Z}^n_{\geq 0},\tilde{v}_{\cdot i}=v_{\cdot i}, \tilde{v}'_{\cdot i}=v_{\cdot i}+D^{-1}(v'_{\cdot i}-v_{\cdot i})$;}

\emph{(2)} {$\tilde{k}_i =k_{i}\prod^{n}_{j=1}d^{v_{ji}}_{j},
\tilde{R}(\tilde{x})=R(x)$.}
\end{definition}

Further, the dynamics of $\tilde{\mathcal{M}}$ is expressed by
\begin{equation}\label{CBP-CRNs}
\dot{\tilde{x}}=\tilde{\Gamma}\tilde{R}(\tilde{x}),
~~~\tilde{x}\in \mathbb{R}_{\geq0}^{n}.
\end{equation}

\begin{remark}\label{re:CBP1}
\cref{df:CBP-CRNs} suggests that $\tilde{\Gamma}=D^{-1}\Gamma$ and $\tilde{x}=D^{-1}x$, from the latter, i.e., $\tilde{x}^{*}=D^{-1}x^{*}$. Moreover, $\tilde{x}^{*}$ is an equilibrium in $\tilde{\mathcal{M}}$ if and only if $x^{*}$ is an equilibrium in $\mathcal{M}$.
\end{remark}

The one-to-one correspondence between $\tilde{x}^{*}$ and $x^{*}$ manifests a momentous property for CBP MASs, as shown below.

\begin{Property}\label{pro:CBP}\emph{(\cite{Wu2020A})}
For any CBP $\tilde{\mathcal{M}}$ generated by a complex balanced  $\mathcal{M}$ under a certain matrix $D$, there is a unique equilibrium in each positive stoichiometric compatibility class.
\end{Property}

\begin{remark}\label{rm:gao}
Substantially, the concept of CBP MAS can be explained by defining a linear transformation, i.e., $\tilde{x}=D^{-1}x$, to bridge the differential equation $\dot{x}=\Gamma R(x)$ to $\dot{\tilde{x}}=\tilde{\Gamma}\tilde{R}(\tilde{x})$. At this point, this notion is consistent with the linear conjugacy concept \cite{Johnston2011Linear,Johnston2012Dynamical}, and a special case of the reconstruction concept \cite{Ke2019Complex}. However, there is a large difference between CBP networks and other twos. The former could stand for practical biochemical networks while the latter twos work as tools and/or even virtual networks.
\end{remark}

We use the following example to exhibit that CBP networks are of practical significance.
\begin{example}\label{eg:CBP}
Consider a class of complex balanced MASs like
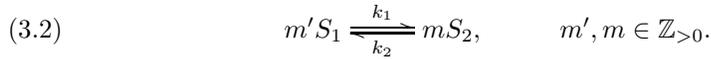
\begin{align}
\xymatrix{m'S_1 \ar @{ -^{>}}^{k_1}  @< 1pt> [r]& mS_2, \ar  @{ -^{>}}^{k_2}  @< 1pt> [l], & m',m \in \mathbb{Z}_{>0}.}
\end{align}
By taking $D=\emph{diag}(m',m)$ we get the CBP MAS in the form of
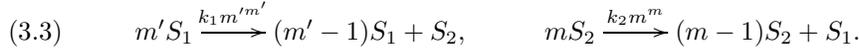
\begin{align}\label{motif:K}
\xymatrix{m'S_1 \ar[r]^-{k_1m'^{m'}}  & (m'-1)S_1+S_2, & mS_2 \ar[r]^-{k_2 m^m} &(m-1)S_2+S_1.}
\end{align}
Actually, this CBP network can correspond to two types of motifs which have been well studied in \cite{NenMotif}. These motifs may be helpful in looking for candidates of biochemical reactions with a small-number
effect for possible biological functions. More precisely,
when $m'\geq 2, m\geq 2$, \cref{motif:K} belongs to motif K and
when $m'=1$ and $m\geq2$, the shape of \cref{motif:K} coincides with motif G.
\end{example}

\subsection{An algorithm for producing CBP CRNs}
\cref{df:CBP-CRNs} will yield a large class of non-weakly reversible CRNs based on a single complex balanced CRN under various matrices $D$'s. The following algorithm gives a systematic way to generate CBP CRNs from a complex balanced CRN.

\begin{algorithm}
	\caption{find all feasible $D=\text{diag}(d_1, \cdots,d_n)$ such that all vectors  $\tilde{v}'_{\cdot i}\in \mathbb{Z}^n_{\geq0}$ for
		$i=1,\cdots,r$, and generate $\tilde{v}'_{\cdot i}$, $\tilde{k}_i$.}
	\label{alg}
	\begin{algorithmic}[1]
		\STATE Input: $v_{\cdot i},v'_{\cdot i},k_i$, $i=1,\cdots,r$
		\FOR {$j=1$ to $n$}{
			\FOR {$i=1$ to $r$}{
				\IF{$v'_{ji}-v_{ji}<0$}
				\STATE $F_{ji}=\left\{\frac{v_{ji}-v'_{ji}}{v_{ji}-a},a=0,\cdots,v_{ji-1}\right\}$
				\ELSE
				\STATE $F_{ji}=\left\{\frac{v_{ji}-v'_{ji}}{v_{ji}-a},a=1,\cdots\right\}$
				\ENDIF}
				\ENDFOR}
			\STATE $F_j=\bigcap F_{ji}$
			\ENDFOR
		\STATE $D=\text{diag}(d_1,\cdots,d_n)$, $d_j\in F_j$.\
		\STATE $\tilde{v}'_{\cdot i}=v_{\cdot i}+D^{-1}(v'_{\cdot i}-v_{\cdot i})$, $\tilde{k}_i =k_{i}\prod^{n}_{j=1}d^{v_{ji}}_{j}$.\
		\STATE Output $\tilde{v}'_{\cdot i}, \tilde{k}_i,D$.
	\end{algorithmic}
\end{algorithm}

The following example exhibits how the algorithm works.

\begin{example}
Consider a modified subnetwork of the Calvin cycle network studied in \cite{GrimbsSpatiotemporal}
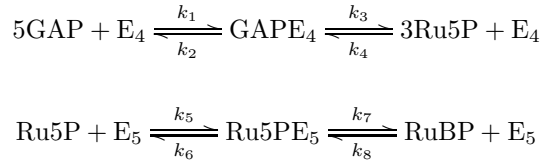
\begin{align*}\label{calvincycle}
\xymatrix{5{\rm GAP}+{\rm E_4} \ar @{ -^{>}}^-{k_1}@< 1pt> [r] & {\rm GAPE_4}\ar  @{ -^{>}}^-{k_2}  @< 1pt> [l] \ar @{ -^{>}}^-{k_3}@< 1pt> [r]& 3{\rm Ru5P}+{\rm E_4} \ar @{ -^{>}}^-{k_4}@< 1pt> [l]\\
{\rm Ru5P}+{\rm E_5} \ar @{ -^{>}}^-{k_5}@< 1pt> [r]& {\rm Ru5PE_5} \ar  @{ -^{>}}^-{k_6}  @< 1pt> [l] \ar @{ -^{>}}^-{k_7}@< 1pt> [r]&{\rm RuBP}+{\rm E_5} \ar @{ -^{>}}^-{k_8}@< 1pt> [l]}
\end{align*}
Let $S_1={\rm GAP}, S_2={\rm GAPE_4}, S_3={\rm Ru5P}, S_4={\rm Ru5PE_5}, S_5={\rm RuBP}$, $S_6={\rm E_4}, S_7={\rm E_5}$. For the specific meanings of these substances, the readers can refer to \cite{GrimbsSpatiotemporal}.
Note that the reactions $3{\rm Ru5P}+{\rm E_4} \rightarrow {\rm GAPE_4}$ and ${\rm RuBP}+{\rm E_5} \rightarrow {\rm Ru5PE_5}$ are additional which do not exist in the real Calvin cycle network.
Suppose this subnetwork is complex balanced, then in terms of the algorithm, we can compute that $d_1\in \{\frac{5}{4},\frac{5}{3},\frac{5}{2},5\}$ while other $d_j=1$ for $j=2,\cdots,7$. Thus it can produce four kinds of CBP CRNs, listed as\\
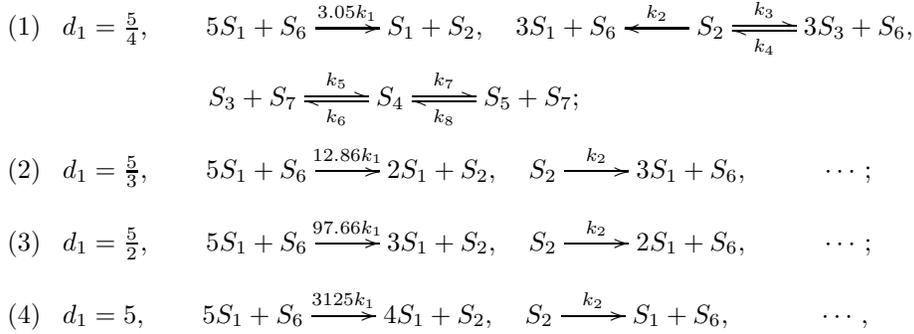
\begin{enumerate}
\item[\emph{(1)}]$\xymatrix{d_1=\frac{5}{4}, \qquad 5S_1+S_6 \ar[r]^-{3.05 k_1}& S_1+S_2,~~~
	3S_1+S_6 &S_2\ar[l]_-{k_2}\ar @{ -^{>}}^-{k_3}@< 1pt> [r]& 3S_3+S_6, \ar @{ -^{>}}^-{k_4}@< 1pt> [l]}$\\
$\xymatrix{~~~~\quad \qquad  \quad S_3+S_7  \ar @{ -^{>}}^-{k_5}@< 1pt> [r]& S_4 \ar  @{ -^{>}}^-{k_6}  @< 1pt> [l] \ar @{ -^{>}}^-{k_7}@< 1pt> [r]&S_5+S_7; \ar @{ -^{>}}^-{k_8}@< 1pt> [l]}$\\
\item[\emph{(2)}]$\xymatrix{d_1=\frac{5}{3}, \qquad 5S_1+S_6 \ar[r]^-{12.86 k_1}& 2S_1+S_2,~~~S_2\ar[r]^-{k_2}&
	3S_1+S_6, &\cdots;}$\\
\item[\emph{(3)}]$\xymatrix{d_1=\frac{5}{2}, \qquad 5S_1+S_6 \ar[r]^-{97.66 k_1}& 3S_1+S_2,~~~S_2\ar[r]^-{k_2}&
	2S_1+S_6, &\cdots;}$\\
\item[\emph{(4)}]$\xymatrix{d_1=5, \qquad 5S_1+S_6 \ar[r]^-{3125 k_1}& 4S_1+S_2,~~~S_2\ar[r]^-{k_2}&
	S_1+S_6, &~\cdots,}$
\end{enumerate}
where we use dots in cases (2), (3), and (4) to represent the remaining reactions that are the same with those reversible reactions emerging in case (1).
\end{example}

\subsection{Lyapunov function PDEs to the stability of CBP MASs}
As stated in \cref{rm:gao}, the asymptotic stability of CBP MASs could be actually addressed through the linear conjugacy or the reconstruction strategy. Hence, the focus should not be on the stability result of CBP MASs itself, but on the alternative way, i.e., the Lyapunov function PDEs, to this result.

For a CBP MAS, defined in \cref{df:CBP-CRNs}, its Lyapunov function PDE, only referring to \cref{eq:LyapunovPDE}, is written as
\begin{equation}\label{cbpeq:LyapunovPDE}
\sum^{r}_{i=1}\tilde{k}_{i}\tilde{x}^{\tilde{v}_{\cdot i}}-\sum^{\tilde{r}}_{i=1}\tilde{k}_{i}\tilde{x}^{\tilde{v}_{\cdot i}}\exp\big\{ (\tilde{v}'_{\cdot i}-\tilde{v}_{\cdot i})^{\top}\nabla f(\tilde{x})\big\}=0.
\end{equation}
We thus have the following stability result for CBP MASs.

\begin{proposition}\emph{(\cite{Wu2020A})}\label{CBP-CRNstability}
For any CBP $\tilde{\mathcal{M}}$ stated in \cref{df:CBP-CRNs}, let $\tilde{x}^*\in {\mathbb{R}^{n}_{>0}}$ be an equilibrium. Then its induced Lyapunov function PDE \cref{cbpeq:LyapunovPDE} could produce a solution in the form of
\begin{equation}\label{eq:ge}
\tilde{G}({\tilde{x}})=\sum^{n}_{j=1}d_j\left(\tilde{x}^*_j-\tilde{x}_j-\tilde{x}_j\ln {\frac{\tilde{x}^*_j}{\tilde{x}_j}}\right)
\end{equation}
as a Lyapunov function to render the locally asymptotic stability of $\tilde{x}^*$ with respect to any initial condition in $\tilde{\mathscr{S}}^+(\tilde{x}^*)$ near $\tilde{x}^*$. Furthermore, if the network is persistent, then $\tilde{x}^*$ is globally asymptotically stable with respect to all initial conditions in $\tilde{\mathscr{S}}^+(\tilde{x}^*)$.
\end{proposition}

We refer to the function $\tilde{G}(\cdot)$ in \cref{eq:ge} as the generalized pseudo-Helmholtz function \cite{Ke2019Complex}.

\begin{example}
Given a complex balanced CRN as
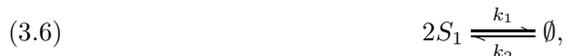
\begin{align}
\xymatrix{2S_1\ar @{ -^{>}}^-{k_1}  @< 1pt> [r] & \emptyset, \ \ar @{ -^{>}}^-{k_2} @< 1pt> [l]}
\end{align}
it possesses a unique positive equilibrium $x^*=\sqrt{\frac{k_1}{k_2}}$. Based on \cref{alg}, we obtain the sole CBP MAS as
\begin{align}\label{eg:bith-death}
\xymatrix{2S_1 \ar[r]^-{4 k_1} & S_1, ~~~ \emptyset \ar[r]^-{k_2}&S_1}
\end{align}
under $D=2$. The CBP MAS corresponds to a typical birth-death process, and has a single equilibrium $\tilde{x}^*=\sqrt{\frac{k_1}{4k_2}}$. From
\cref{CBP-CRNstability}, it is straightforward to know that the Lyapunov function PDE of \cref{cbpeq:LyapunovPDE} for this CBP MAS admits a solution
$$\tilde{G}(\tilde{x})=2\left(\sqrt{\frac{k_1}{4k_2}}-\tilde{x}-\tilde{x}
\ln\sqrt{\frac{k_1}{4k_2}}+\tilde{x}\ln \tilde{x}\right)$$
to behave as a Lyapunov function rendering the local asymptotic stability of $\tilde{x}^*$.
\end{example}

It should be noted that the birth-death processes have been studied well from the viewpoint of microscopic level. Anderson \cite{Anderson2015Lyapunov} proposed the scaling limit of the non-equilibrium potential as a Lyapunov function to capture asymptotic stability. This example illustrates some specific biological systems might be analyzed on dynamical behaviors from the viewpoint of CBP MASs.



\section{Stability of MASs compounded of a CBP MAS and a series of MASs of dim$\mathscr{S}=1$} \label{sec4}
In this section, we will consider a class of MASs consisting of a CBP MAS and a series of MASs of dim$\mathscr{S}=1$, called CBP-$\ell$Sub1 MASs in the context. Like Com-$\ell$Sub1 MASs, we set all subnetworks in CBP-$\ell$Sub1 MASs to be mutually independent according to species.


\subsection{Revisiting stability of any $1$-dimensional MAS through Lyapunov function PDEs}
In \cref{sec2.2}, we have listed the Lyapunov function, i.e., \cref{sub1Lya}, as a solution of the corresponding PDEs for any $1$-dimensional MAS. Here, we give more details for use.

\begin{theorem}[\cite{Fang2015Lyapunov}]\label{thm:sub1}
For any $1$-dimensional $\mathcal{M}$ with an equilibrium ${x}{^*}\in \mathbb{R}^{n}_{>0}$, let
\begin{align}\label{eq:sub1Lya}
h({x},{u})=\sum_{\{i|\beta_{i} >0\}}({k}_i {x}^{{v}_{\cdot i}})\bigg(\sum^{\beta_{i}-1}_{j=0} {u}^j \bigg)
	+\sum_{\{i|\beta_{i} <0\}}({k}_i  {x}^{{v}_{\cdot i}})\bigg(-\sum^{-1}_{j=\beta_{i}} {u}^j \bigg),
	\end{align}
	where ${u}=\exp\{\omega^\top \nabla{f}\}$, $\omega \in \mathbb{R}^{n}\setminus \{\mathbbold{0}_{n}\}$ represents a set of bases of ${\mathscr{S}}$, and $\beta_{i}\in \mathbb{Z}\setminus\{0\}~satisfies~{v}'_{\cdot i}-{v}_{\cdot i}=\beta_{i} \omega, ~i=1,\cdots,r$. Then the Lyapunov function PDEs \cref{eq:LyapunovPDE} and \cref{eq:boundary} of this $\mathcal{M}$ have a solution in the form of
$$f({x})=\int^{\gamma({x})}_0 \ln \tilde{u}(y^{\dag}({x})+\alpha \omega)\dd \alpha$$
that can behave as a Lyapunov function. Here, $\tilde{u}$ makes $h({x},{u})=0$, $\gamma \in\mathscr{C}^2(\mathbb{R}^n_{>0};\mathbb{R}_{>0})$ and $y^{\dag}\in\mathscr{C}^2(\mathbb{R}^n_{>0};\mathbb{R}^n_{>0})$ gratify $x=y^{\dag}(x)+\gamma(x)\omega$ and
	$\gamma(x+\delta \omega)=\gamma(x)+\delta$ $\forall \delta\in \mathbb{R}$, respectively.
	
Further, if $\omega^\top \frac{\partial}{\partial{x}} h({x}^{*},1)<0$, this MAS is locally asymptotically stable at ${x}^{*}$.

\end{theorem}

The above solution plays an important role on constructing solutions of the Lyapunov function PDEs for $1$-dimensional MASs compounded with CBP MASs.

\subsection{Lyapunov function PDE to the stability of CBP-$\ell$Sub1 MASs}
In this subsection we work out an appropriate Lyapunov function for CBP-$\ell$Sub1 MASs to character the convergent behavior by exploiting the Lyapunov function PDE method.

Consider a CBP-$\ell$Sub1 $\tilde{\mathcal{M}}$ involved with a CBP MAS coming from \cref{df:CBP-CRNs}, represented as $\tilde{\mathcal{M}}^{(0)}$ and some $1$-dimensional MASs, labeled as $\tilde{\mathcal{M}}^{(p)}|^{\ell}_{p=1}$. Note that all of these subnetworks are mutually independent, that is $\tilde{\mathcal{S}}^{(p)}\bigcap\tilde{\mathcal{S}}^{(q)}=\emptyset$, $\forall p,q\in \{0,1,\cdots,\ell\}$.

In the next, we denote
\begin{align*}
n=\sum^{\ell}_{p=0}n_{p},~
\tilde{v}^{(p)}_{\cdot i}=\bigotimes^{\iota -1}_{q=0} \mathbbold{0}_{n_q}\bigotimes
\tilde{v}_{\cdot i(p)}\bigotimes^{\ell}_{q=p+1}\mathbbold{0}_{n_q},~
\tilde{v}'^{(p)}_{\cdot i}=\bigotimes^{p -1}_{q=0} \mathbbold{0}_{n_q}\bigotimes
\tilde{v}'_{\cdot i(p)}\bigotimes^{\ell}_{q=p+1}\mathbbold{0}_{n_q},
\end{align*}
where $n_{p}, r_{p}, \tilde{v}_{\cdot i({p})}$, and $\tilde{v}'_{\cdot i(p)}$ stand for the number of species and of reactions of the $p$th subnetwork, the reactant complex as well as the resultant complex of the $i$th reaction, respectively; $\bigotimes$ is the Cartesian product; the term $n_q=0$ if $q<0$ or $q>\ell$. Immediately, we get the compound $\tilde{\mathcal{M}}$ with $$\tilde{\mathcal{S}}= \bigcup^\ell_{p=0} \tilde{\mathcal{S}}^{(p)},
~~~~
\tilde{\mathcal{C}}=
\bigcup^{\ell}_{p=0}
\bigcup^{r_{p}}_{i=1}\{\tilde{v}^{(p)}_{\cdot i}, \tilde{v}'^{(p)}_{\cdot i}\},
~~~~
\tilde{\mathcal{R}}=
\bigcup^{\ell}_{p=0}
\bigcup^{r_{p}}_{i=1}
\{\tilde{v}^{(p)}_{\cdot i} \stackrel{\tilde{k}^{(p)}_{i}}\longrightarrow \tilde{v}'^{(p)}_{\cdot i}\},$$
and the stoichiometric subspace $$\tilde{\mathscr{S}}=\bigotimes^{\ell}_{p=0}\tilde{\mathscr{S}}^{(p)},$$ where $\tilde{\mathscr{S}}^{(p)}$ represents the respective stoichiometric subspace of each subsystem.
Besides, the dynamics of $\tilde{\mathcal{M}}$ is as follows:
\begin{align}\label{eq:dynamicCBP-sub1}
\dot{\tilde{x}}=\sum^{\ell}_{p=0} \sum^{r_{p}}_{i=1}\tilde{k}^{(p)}_i \tilde{x}^{\tilde{v}^{(p)}_{\cdot i}}
\bigg(\tilde{v}'^{(p)}_{\cdot i}-\tilde{v}^{(p)}_{\cdot i}\bigg),
\end{align}
where $\tilde{x}=\bigotimes^{\ell}_{p=0}\tilde{x}^{(p)}$ is the state of the CBP-$\ell$Sub1 MAS.

Following with the above information, we write easily the Lyapunov function PDE for this CBP-$\ell$Sub1 $\tilde{\mathcal{M}}$ to be



\begin{align}\label{eq:CBP-sub1PDE}
	\sum^{\ell}_{p=0}
	\sum^{r_p}_{i=1}
	\left(
	\tilde{k}^{(p)}_{i} \tilde{x}^{(p)\tilde{v}_{\cdot i (p)}}-
	\tilde{k}^{(p)}_{i} \tilde{x}^{(p)\tilde{v}_{\cdot i (p)}}\exp \left\{(\tilde{v}'_{\cdot i(p)}-\tilde{v}_{\cdot i(p)})^\top \frac{\partial{f(\tilde{x})}}{\partial \tilde{x}^{(p)}}\right \}\right)=0.
	\end{align}
	
\begin{lemma}\label{lemmaGao}
For a CBP-$\ell$Sub1 $\tilde{\mathcal{M}}$ composed of a CBP MAS $\tilde{\mathcal{M}}^{(0)}$ and $\ell$ $1$-dimensional MASs $\tilde{\mathcal{M}}^{(p)}|^{\ell}_{p=1}$, the dynamics follows \cref{eq:dynamicCBP-sub1} and $\tilde{x}^*\in\mathbb{R}^n_{>0}$ is an equilibrium, then its Lyapunov function PDE of \cref{eq:CBP-sub1PDE} admits a solution
	\begin{align}\label{CBP-sub1Lya}
	f(\tilde{x})=\sum^{n_0}_{i=1}d_i\bigg (\tilde{x}^{{(0)}^*}_{i}-\tilde{x}^{(0)}_{i}-\tilde{x}^{(0)}_{i}
	\ln \frac{\tilde{x}^{{(0)}^*}_{i}}{\tilde{x}^{(0)}_{i}}\bigg)+\sum^{\ell}_{p=1}\int^{\gamma_p(\tilde{x}^{(p)})}_{0} \ln \tilde{u}^{(p)}(y^{\dag}(\tilde{x}^{(p)})+\alpha \omega_p)\text{d}\alpha,
	\end{align}
where $\gamma_p(\cdot),\tilde{u}^{(p)},y^{\dag}(\cdot),\omega_p$ share the same meanings with $\gamma(\cdot),\tilde{u},y^{\dag}(\cdot),\omega$ in \cref{sub1Lya}, respectively. 	
\end{lemma}
\begin{proof}
The proof is given in Appendix A.
\end{proof}

Then we could reach the asymptotic stability of CBP-$\ell$Sub1 MASs through the Lyapunov function PDEs method.
\begin{theorem}\label{thm:CBP-sub1}
For a CBP-$\ell$Sub1 $\tilde{\mathcal{M}}$ defined as in \cref{lemmaGao}, let $\tilde{x}^*=\bigotimes^{\ell}_{p=0}\tilde{x}^{(p)}$ be a positive equilibrium in $\tilde{\mathcal{M}}$. Then for every $1$-dimensional $\tilde{\mathcal{M}}^{(p)}~(p=1,\cdots,\ell)$ if $\omega^\top_p \frac{\partial}{\partial\tilde{x}^{(p)}} h_p(\tilde{x}^{(p)^*},1)<0$, $\tilde{x}^*$ is locally asymptotically stable, where $\omega_p,h_p(\cdot,1)$ follow the same meanings with $\omega,h(\cdot,1)$ in \cref{thm:sub1}, respectively.	


\end{theorem}

\begin{proof}
The detailed proof can be found in Appendix A.
\end{proof}

An example is given to illustrate the availability of Lyapunov function PDE for CBP-$\ell$Sub1 MASs.

\begin{example}
 Given a CBP-$\ell$Sub1 $\tilde{\mathcal{M}}$ (${\ell}=1$) as follows
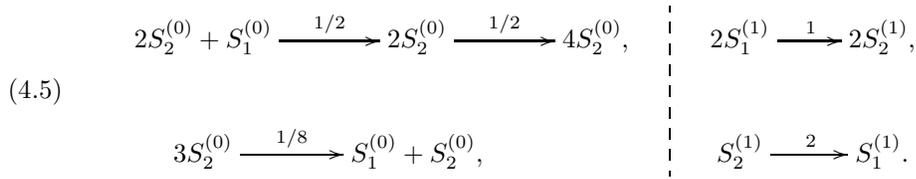
\begin{equation}\label{eg4}
\begin{array}{c:c}
~~\xymatrix{     2S^{(0)}_2+S^{(0)}_1 \ar[r]^-{1/2}         & 2S^{(0)}_2\ar[r]^-{1/2}  &4S^{(0)}_2,            \\
3S^{(0)}_2\ar[r]^-{1/8}  &S^{(0)}_1+S^{(0)}_2, &  }
~~&~~
\xymatrix{ 2S^{(1)}_1  \ar[r]^{1} &2S^{(1)}_2,\\
	S^{(1)}_2\ar[r]^{2} &S^{(1)}_1.}~~
\end{array}
\end{equation}
where the left part is a CBP $\tilde{\mathcal{M}}^{(0)}$ generated from the following complex balanced MAS
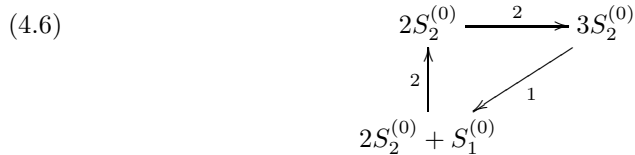
\begin{align}\label{eg:cb}
\xymatrix{
	2S^{(0)}_2 \ar[r]^-{2}  &  3S^{(0)}_2 \ar[dl]^-{1}    \\
	2S^{(0)}_2+S^{(0)}_1 \ar[u]^-{2} &   }
\end{align}
under $D=(1,\frac{1}{2})$ while the right part is a $1$-dimensional $\tilde{\mathcal{M}}^{(1)}$. Moreover, these two subnetworks are mutually independent according to species. This $\tilde{\mathcal{M}}$ is $3$-dimensional and of deficiency $2$. It is easy to compute an equilibrium to be $\tilde{x}^{*}=(1,4,1,1)^\top$ if the positive compatibility class of $\tilde{\mathcal{M}}^{(1)}$ is selected as $\{\tilde{x}_1^{(1)}+\tilde{x}_2^{(1)}=2\}$. The complexes are
\begin{align*}
&\tilde{v}_{\cdot 1(0)}=(1,2)^\top,
\tilde{v}'_{\cdot 1(0)}=\tilde{v}_{\cdot 2(0)}=(0,2)^\top,
\tilde{v}'_{\cdot 2(0)}=(0,4)^\top,
\tilde{v}_{\cdot 3(0)}=(0,3)^\top,\notag\\
&\tilde{v}'_{\cdot 3(0)}=(1,1)^\top,\tilde{v}_{\cdot 1(1)} =(2,0)^\top,
\tilde{v}'_{\cdot 1(1)} =(0,2)^\top,
\tilde{v}_{\cdot 2(1)} =(0,1)^\top,
\tilde{v}'_{\cdot 2(1)} =(1,0)^\top.
\end{align*}
From \cref{lemmaGao}, we get a solution for the Lyapunov function PDE of this CBP-$\ell$Sub1 $\mathcal{M}$ to be
$$
f(\tilde{x})=3-\tilde{x}^{(0)}_{1}-\tilde{x}^{(0)}_{1}\ln \tilde{x}^{(0)}_1-\frac{1}{2}\tilde{x}^{(0)}_2-\frac{1}{2}\tilde{x}^{(0)}_2
\ln \frac{4}{\tilde{x}^{(0)}_{2}}+\int^{\gamma_1(\tilde{x}^{(1)})}_{0} \ln \tilde{u}^{(1)}(y^{\dag}(\tilde{x}^{(1)})+\alpha \omega_1)\text{d}\alpha,$$
where
$\omega_1=(-1,1)^\top,~\gamma_1(\tilde{x}^{(1)})
=\frac{\tilde{x}^{(1)}_2-\tilde{x}^{(1)}_1}{2}$,
$\tilde{u}^{(1)}(\tilde{x}^{(1)})=\frac{-\tilde{x}^{{(1)}^2}_1+\tilde{x}^{(1)}_1\sqrt{\tilde{x}^{{(1)}^2}_1+8\tilde{x}^{(1)}_2}}{2\tilde{x}^{{(1)}^2}_1}$, and $y^{\dag}(\tilde{x}^{(1)})=\left(\frac{\tilde{x}^{(1)}_1+\tilde{x}^{(1)}_2}{2},\frac{\tilde{x}^{(1)}_1+\tilde{x}^{(1)}_2}{2}\right)^\top$. Based on \cref{eq:sub1Lya}, we have $$h_1(\tilde{x}^{(1)},u^{(1)})=\tilde{x}^{{(1)}^2}_1(1+u^{(1)})-2\tilde{x}^{(1)}_2u^{(1)^{-1}}$$ that naturally supports
$$\omega^\top_1 \frac{\partial}{\partial\tilde{x}^{(1)}} h_1(\tilde{x}^{(1)},1)|_{\tilde{x}^{(1)}
	=\tilde{x}^{(1)^*}}=-4\tilde{x}^{{(1)}^*}_1-2<0.$$
Therefore, the asymptotic stability of $\tilde{x}^*=(1,4,1,1)^\top$ is achieved according to \cref{thm:CBP-sub1}.

\end{example}

\section{Stablity of MASs compounded of a CBP MAS and a few autocatalytic MASs} \label{sec5}
We have successfully worked out CBP-$\ell$Sub1 MASs on asymptotic stability through the Lyapunov function PDEs method. However, a requirement is that all subnetworks in a CBP-$\ell$Sub1 MAS are mutually independent according to species. This restriction leads to the solution of the PDE of a CBP-$\ell$Sub1 MAS able to be constructed by a combination of solution of the corresponding PDE of every subnetwork. The case will become complicated if the species among all subnetworks is not independent. We follow this issue in this section by defining a kind of MASs compounded of a CBP MAS and a few $1$-dimensional autocatalytic MASs.



\subsection{Compound MASs of a CBP MAS and autocatalytic MASs}
Autocatalytic reactions are ubiquitous in living organisms, like metabolism, DNA replications, etc. Generally speaking, they refer to a class of reactions where the reactants themselves act as catalysts. There are various expressions with respect to this notion \cite{hoessly2019stationary,Hordijk2004,Gopalkrishnan2011}, and one of them is as follows.

\begin{definition}[autocatalystic MAS, a reduced version of \cite{hoessly2019stationary}]\label{autoca}
	A MAS is said to be an autocatalytic one, labeled by $\mathcal{M}=(\mathcal{S,C,R,K})$, if the following conditions are true
	\begin{enumerate}
		\item[\emph{(1)}] all reactions have a net consumption of one $S_i$ and a net production one $S_j$, i.e., in the form of
		\begin{equation*}
		\xymatrix{S_i+(m-1)S_j \ar[r]^-{k_m}  & m S_j},
		\end{equation*}
		where $m\geq1,i,j= 1,\cdots, n$;
		
		\item[\emph{(2)}] there is one monomolecular linkage class;
		
		\item[\emph{(3)}] if there is a net consumption of one $S_i$ and a net producing one $S_j$ in a reaction, then $S_i \longrightarrow S_j, S_j \longrightarrow S_i \in \mathcal{R}$, which means mass exchange in both directions only happens in single molecular reactions.
	\end{enumerate}
\end{definition}

\begin{remark}\label{re:auto}
From \cref{autoca}, it is obvious that if an autocatalytic MAS only has two species, then its stoichiometric subspace is $1$-dimensional. We name this class of networks two-species autocatalytic ones, which are our main concerns in the subsequent investigation.
\end{remark}

Consider an $\tilde{\mathcal{M}}$ composed of a CBP $\tilde{\mathcal{M}}^{(0)}=(\tilde{\mathcal{S}}^{(0)},\tilde{\mathcal{C}}^{(0)},\tilde{\mathcal{R}}^{(0)},\tilde{\mathcal{K}}^{(0)})$ as stated in \cref{df:CBP-CRNs}, and a number of two-species autocatalytic MASs in the form of
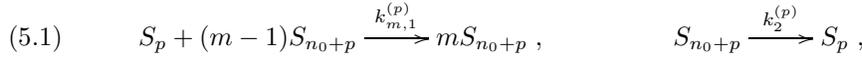
\begin{align}\label{eq:autocatalytic}
&\xymatrix{S_p+(m-1)S_{n_0+p} \ar[r]^-{k^{(p)}_{m,1}}  & m S_{n_0+p}},
&\xymatrix{S_{n_0+p} \ar[r]^-{k^{(p)}_{2}}  &  S_{p}},
\end{align}
labeled by $\tilde{\mathcal{M}}^{(p)}=(\tilde{\mathcal{S}}^{(p)},\tilde{\mathcal{C}}^{(p)},\tilde{\mathcal{R}}^{(p)},\tilde{\mathcal{K}}^{(p)}),~p=1,\cdots,\ell$, where $m \in \mathcal{I}'_p\subseteq\mathcal{I}_p=\{1,\cdots,\tau_p\}$ and $1\in\mathcal{I}'_p$. Here, we denote $\tilde{\mathcal{S}}^{(0)}=\{S_1,\cdots,S_{n_0}\}$ and $\tilde{\mathcal{S}}^{(p)}=\{S_p,S_{n_0+p}\}$ with $p=1,\cdots,\ell$. We call the above $\tilde{\mathcal{M}}$ a CBP-$\ell$ts-Autoca MAS, which naturally meets
(i) $n_0\geq {\ell}$; (ii) $ \forall i,j \in \{1,\cdots,{\ell}\}, ~\tilde{\mathcal{S}}^{(i)} \bigcap \tilde{\mathcal{S}}^{(j)}=\emptyset$ while $\tilde{\mathcal{S}}^{(0)} \bigcap \tilde{\mathcal{S}}^{(p)}=S_p$, $\forall\{p\}^{\ell}_{p=1}$.

By taking the same notations as in CBP-$\ell$Sub1 MASs, we use $n_{p}$ and $r_{p}$ to represent the numbers of species and reactions, and $\tilde{v}_{\cdot i(p)}$ as well as $\tilde{v}'_{\cdot i(p)}$ to represent the reactant complex and the resultant complex of the $i$th reaction ($i=1,\cdots,r_{p}$) of every subnetwork ($p=0,\cdots,\ell$) in a CBP-$\ell$ts-Autoca MAS, respectively. Thus we immediately obtain $n=n_0+{\ell}$, $\tilde{v}^{{(0)}^\top}_{\cdot i}=(\tilde{v}^\top_{\cdot i(0)},\mathbbold{0}^\top_{\ell}),\tilde{v}'^{{(0)}^\top}_{\cdot i}=(\tilde{v}'^\top_{\cdot i(0)},\mathbbold{0}^\top_{\ell})$, and $\tilde{v}^{{(p)}^\top}_{\cdot i}=(\mathbbold{0}^\top_{p-1}, \tilde{v}_{1i(p)}, \mathbbold{0}^\top_{n_0-1}, \tilde{v}_{2i(p)}, \mathbbold{0}^\top_{{\ell}-p}),$ $\tilde{v}'^{{(p)}^\top}_{\cdot i}=(\mathbbold{0}^\top_{p-1}, \tilde{v}'_{1i(p)}, \mathbbold{0}^\top_{n_0-1}, \tilde{v}'_{2i(p)}, \mathbbold{0}^\top_{{\ell}-p})$ when $p=1, \cdots, {\ell}$.
The dynamics follows the same expression as \cref{eq:dynamicCBP-sub1}, i.e.,
\begin{equation}\label{eq:dynamicCBP-auto}
\dot{\tilde{x}}=\sum^{\ell}_{p=0}\sum^{r_
	{p}}_{i=1}
\tilde{k}^{(p)}_i \tilde{x}^{\tilde{v}^{(p)}_{\cdot i}}\left(\tilde{v}'^{(p)}_{\cdot i}-\tilde{v}^{(p)}_{\cdot i}\right),
\end{equation}
but with all variables defined as for a CBP-$\ell$ts-Autoca MAS.

From \cref{equilibrium}, a concentration vector
\begin{equation}\label{equilibriumAuto}
\tilde{x}^*=(\tilde{x}_1^*,\cdots, \tilde{x}^*_{n_0},\tilde{x}^*_{n_0+1},\cdots, \tilde{x}_{n_0+{\ell}}^*)^\top \in \mathbb{R}^{n}_{> 0}
\end{equation}
is a positive equilibrium in the CBP-$\ell$ts-Autoca $\tilde{\mathcal{M}}$ if $\dot{\tilde{x}}=0$ evaluated at $\tilde{x}=\tilde{x}^*$. For every positive equilibrium in the CBP-$\ell$ts-Autoca $\tilde{\mathcal{M}}$, we have the following property.

\begin{Property}\label{pro:3}
For a CBP-$\ell$ts-Autoca $\tilde{\mathcal{M}}$ modelled by \cref{eq:dynamicCBP-auto}, a concentration vector $\tilde{x}^*\in\mathbb{R}^n_{>0}$ given by \cref{equilibriumAuto} is an equilibrium in $\tilde{\mathcal{M}}$ if and only if $\tilde{x}^{{(0)}^*}=(\tilde{x}_1^*,\cdots,\tilde{x}^*_{n_0})^\top$ is a positive equilibrium in $\tilde{\mathcal{M}}^{(0)}$ while $\tilde{x}^{{(p)}^*}=(\tilde{x}_p^*,\tilde{x}_{n_0+p}^*)^\top$ is a reaction vector balanced equilibrium in $\tilde{\mathcal{M}}^{(p)}$ for $p=1,\cdots,\ell$.
\end{Property}

\begin{proof}
	The detailed proof can be found in Appendix B.
\end{proof}

We characterize the number of positive equilibria of a CBP-$\ell$ts-Autoca MAS in each positive stoichiometric compatibility class through the following lemma.

\begin{lemma}\label{le:CBP-autoequilibriumunique}
Given a CBP-$\ell$ts-Autoca $\tilde{\mathcal{M}}$, ruled by \cref{eq:dynamicCBP-auto} and admitting an equilibrium $\tilde{x}^*\in\mathbb{R}^n_{>0}$ defined in \cref{equilibriumAuto}, each positive stoichiometric compatibility class contains at most a positive equilibrium if one of the following conditions holds:
\begin{enumerate}
    \item[\emph{(1)}] $\forall \{p\}^{\ell}_{p=1}$ and $\forall m$ in \cref{eq:autocatalytic} with $m \in \mathcal{I}'_p\subseteq\mathcal{I}_p=\{1,\cdots,\tau_p\}$ and $1\in \mathcal{I}'_p$, $\tau_p\leq2$; or
\item[\emph{(2)}] $\forall \{p\}^{\ell}_{p=1}$ and $\forall m$ in \cref{eq:autocatalytic} with $m \in \mathcal{I}'_p\subseteq\mathcal{I}_p=\{1,\cdots,\tau_p\}$ and $1 \in \mathcal{I}'_p$, $\tilde{\mathcal{M}}^{(p)}$ is mass-conserved if the corresponding subset $\mathcal{I}'_p$ contains one element greater than $2$.
\end{enumerate}
	
	
\end{lemma}

\begin{proof}
	The proof can be caught in Appendix B.	
\end{proof}

\subsection{Lyapunov function PDE to the stability of CBP-$\ell$ts-Autoca MASs}
In this subsection, we capture the asymptotic stability of CBP-$\ell$ts-Autoca MASs also through the Lyapunov function PDEs strategy.

From the definition of CBP-$\ell$ts-Autoca MASs, it is easy to write out the corresponding Lyapunov function PDE, which has the same portrait as \cref{eq:CBP-sub1PDE} but with all variables defined as for a CBP-$\ell$ts-Autoca MAS.

\begin{lemma}\label{le:CBP-autoLya}
For a CBP-$\ell$ts-Autoca $\tilde{\mathcal{M}}$ governed by \cref{eq:dynamicCBP-auto} and possessing an equilibrium $\tilde{x}^*$ defined by \cref{equilibriumAuto}, the twice differentiable function
\begin{align}\label{eq:CBP-autoLya}
f(\tilde{x})=\sum^{n_0}_{i=1}d_i\bigg (\tilde{x}^*_{i}-\tilde{x}_{i}-\tilde{x}_{i}
\ln \frac{\tilde{x}^*_{i}}{\tilde{x}_{i}}\bigg)
+\sum^{\ell}_{p=1}\int^{\tilde{x}_{n_0+p}}_{\tilde{x}^*_{n_0+p}}
\ln \frac{{k}^{(p)}_{2}\alpha_p}{\sum _{m\in \mathcal{I}'_p}
{k}^{(p)}_{m,1} \tilde{x}^*_{p} \alpha^{m-1}_{p}}d \alpha_p,
\end{align}
with $d_p=1$ from $p=1$ to $\ell$ is a solution of the Lyapunov function PDE \emph{(}\cref{eq:CBP-sub1PDE}-like equation\emph{)} induced by $\tilde{\mathcal{M}}$.
\end{lemma}

\begin{proof}
The detailed proof is given in Appendix B.
\end{proof}


Then we can reach the asymptotic stability of CBP-$\ell$ts-Autoca MASs based on the above results.

\begin{theorem}\label{thm:CBPauto}
For a CBP-$\ell$ts-Autoca $\tilde{\mathcal{M}}$ described by \cref{eq:dynamicCBP-auto} and admitting an equilibrium $\tilde{x}^* \in \mathbb{R}^n_{>0}$ defined by \cref{equilibriumAuto}, $\tilde{x}^*$ is locally asymptotically stable
\begin{enumerate}
    \item[\emph{(1)}] if $\forall \{p\}^{\ell}_{p=1}$ and $\forall m$ in \cref{eq:autocatalytic} with $m \in \mathcal{I}'_p\subseteq\mathcal{I}_p=\{1,\cdots,\tau_p\}$ and $1\in\mathcal{I}'_p$, $\tau_p\leq 2$; or
\item[\emph{(2)}] if $\forall \{p\}^{\ell}_{p=1}$ when $\mathcal{I}'_p$ contains one element greater than $2$, $\tilde{\mathcal{M}}^{(p)}$ is mass-conserved and $\sum_{m\in \mathcal{I}'_p}(2-m){k}^{(p)}_{m,1}
\tilde{x}^{* m-1}_{n_0+p}>0$.
\end{enumerate}
\end{theorem}

\begin{proof}
The detailed proof can be found in Appendix B.
\end{proof}

The following two examples serve for illustrating the validity of the Lyapunov function PDEs way in CBP-$\ell$ts-Autoca MASs.

\begin{example}
Consider a CBP-$\ell$ts-Autoca $\tilde{\mathcal{M}}$ ($\ell=1$) with the reaction route
$$
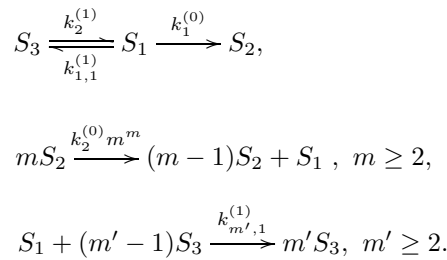
\xymatrix{S_3 \ar @{ -^{>}}^-{k^{(1)}_2}  @< 1pt> [r]& S_1 \ar  @{ -^{>}}^-{k^{(1)}_{1,1}}  @< 1pt> [l] \ar[r]^-{k^{(0)}_1} &S_2,}$$
$$\qquad \qquad \qquad \xymatrix{mS_2\ar[r]^-{k^{(0)}_2 m^m}&(m-1)S_2+S_1},~m\geq2,$$
$$\qquad \qquad \quad\quad\quad \xymatrix{S_1+(m'-1)S_3\ar[r]^-{k^{(1)}_{m',1}}&m'S_3,~m'\geq2.}$$
This $\tilde{\mathcal{M}}$ is non-weakly-reversible, $2$-dimensional and of deficiency $2$.
Note that the present CBP $\tilde{\mathcal{M}}^{(0)}$ (identified by the reaction rate constants with superscript $(0)$) and the autocatalytic $\tilde{\mathcal{M}}^{(1)}$ correspond to motif G (see \cref{eg:CBP}) and motif F given in \cite{NenMotif}, respectively. As a matter of fact, $\tilde{\mathcal{M}}$ can be viewed as a special case of the compound network of motif G and motif F.

Let $m=2$, $m'=2$, i.e., $\tau_1=2$ and $\mathcal{I}'_1=\{1,2\}$, and $k^{(1)}_{1,1}>k^{(1)}_{2,1}k^{(0)}_2$, then $\tilde{\mathcal{M}}^{(0)}$ admits a positive equilibrium $(\tilde{x}^*_1,\tilde{x}^*_2)=\left(k^{(0)}_2,\frac{1}{2}\sqrt{k^{(0)}_1}\right)$ and has complexes as
 \begin{align*}
 \tilde{v}_{\cdot 1(0)}=(1,0)^\top,
 \tilde{v}'_{\cdot 1(0)}=(0,1)^\top,
 \tilde{v}_{\cdot 2(0)}=(0,2)^\top,
 \tilde{v}'_{\cdot 2(0)}=(1,1)^\top
 \end{align*}
while $\tilde{\mathcal{M}}^{(1)}$ admits a reaction vector balanced equilibrium $$(\tilde{x}^*_1,\tilde{x}^*_3)=\left(k^{(0)}_2, \frac{k^{(1)}_{2} k^{(0)}_2}{k^{(1)}_{1,1}-k^{(1)}_{2,1}k^{(0)}_2}\right)$$ and possesses complexes as
\begin{align*}
\tilde{v}_{\cdot 1(1)}=\tilde{v}'_{\cdot 3(1)}=(1,0)^\top,
\tilde{v}'_{\cdot 1(1)}=\tilde{v}_{\cdot 3(1)}=(0,1)^\top,
\tilde{v}_{\cdot 2(1)}=(1,1)^\top,
\tilde{v}'_{\cdot 2(1)}=(0,2)^\top.
\end{align*}
According to \cref{le:CBP-autoLya}, the corresponding Lyapunov function PDE for $\tilde{\mathcal{M}}$ admits a solution in the form of
\begin{align*}
&f(\tilde{x})=\tilde{x}^*_{1}-\tilde{x}_{1}-\tilde{x}_{1}
\ln \frac{\tilde{x}^*_{1}}{\tilde{x}_{1}}+2\bigg(\tilde{x}^*_{2}-\tilde{x}_{2}-\tilde{x}_{2}
\ln \frac{\tilde{x}^*_{2}}{\tilde{x}_{2}}\bigg)+\tilde{x}_3 \ln ({k^{(1)}_2\tilde{x}_3})-\tilde{x}^*_3 \ln ({k^{(1)}_2\tilde{x}^*_3})-\\\notag
&
k^{(1)^{-1}}_{2,1}\bigg[(k^{(1)}_{1,1}+k^{(1)}_{2,1}\tilde{x}_3)\ln(\tilde{x}^*_1k^{(1)}_{1,1}+\tilde{x}^*_1k^{(1)}_{2,1} \tilde{x}_3)-(k^{(1)}_{1,1}+k^{(1)}_{2,1}\tilde{x}^*_3)\ln(\tilde{x}^*_1k^{(1)}_{1,1}+\tilde{x}^*_1k^{(1)}_{2,1} \tilde{x}^*_3)\bigg]
\end{align*}
It is not hard to compute the Hessian matrix of $f(\tilde{x})$ to be
\begin{align*}
\nabla^2 f(\tilde{x})=\emph{diag}\bigg(\tilde{x}^{-1}_1, 2\tilde{x}^{-1}_2, \frac{k^{(1)}_{1,1}}{\tilde{x}_3(k^{(1)}_{1,1}+k^{(1)}_{2,1}\tilde{x}_3)}\bigg),
\end{align*}
which is obviously strictly convex. From \cref{thm:CBPauto} or \cref{asymstability}, $f(\tilde{x})$ is qualified as a Lyapunov function to suggest the asymptotic stability of $(\tilde{x}^*_1, \tilde{x}^*_2, \tilde{x}^*_3)^\top$.
\end{example}

\begin{example}
Consider a CBP-$\ell$ts-Autoca $\tilde{\mathcal{M}}$ (${\ell}=1$) with $\text{dim}\tilde{\mathscr{S}}=3$ and deficiency $4$, as can be seen below, including a CBP $\tilde{\mathcal{M}}^{(0)}$ in the left hand side and an autocatalytic $\tilde{\mathcal{M}}^{(1)}$ in the right hand side,
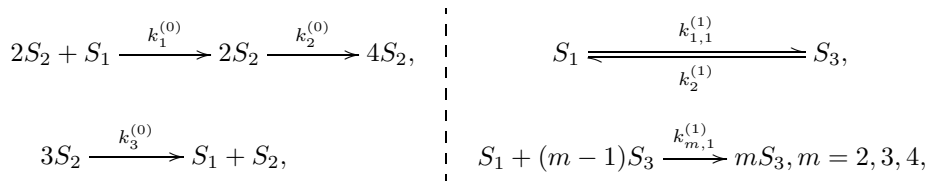
\begin{equation*}
\begin{array}{c:c}
~\xymatrix{     2S_2+S_1 \ar[r]^-{k^{(0)}_1}         & 2S_2\ar[r]^-{k^{(0)}_2}  &4S_2,            \\
	3S_2\ar[r]^-{k^{(0)}_3}  &S_1+S_2, &  }
~&~
\xymatrix{ S_1\ar @{ -^{>}}^-{k^{(1)}_{1,1}}  @< 1pt> [r]&S_3, \ar  @{ -^{>}}^-{k^{(1)}_2}  @< 1pt> [l]\\
	S_1+(m-1)S_3\ar[r]^-{k^{(1)}_{m,1}} &mS_3, m=2,3,4,}
\end{array}
\end{equation*}
where $k^{(0)}_1=\frac{1}{2}, k^{(0)}_1=\frac{1}{2}, k^{(0)}_3=\frac{1}{8}$, and $k^{(1)}_{1,1}=8, k^{(1)}_{2,1}=2, k^{(1)}_{3,1}=1, k^{(1)}_{4,1}=1, k^{(1)}_2=12.$
Note that the CBP $\tilde{\mathcal{M}}^{(0)}$ is generated by the complex balanced MAS given in \cref{eg:cb} under $D=\emph{diag}(1,\frac{1}{2})$, and possesses an equilibrium $\tilde{x}^{(0)^*}=(1,4)^\top$ and the complexes
\begin{align*}
\tilde{v}_{\cdot 1(0)}=(1,2)^\top,
\tilde{v}'_{\cdot 1(0)}=\tilde{v}_{\cdot 2(0)}=(0,2)^\top,
\tilde{v}'_{\cdot 2(0)}=(0,4)^\top,
\tilde{v}_{\cdot 3(0)}=(0,3)^\top,
\tilde{v}'_{\cdot 3(0)}=(1,1)^\top.
\end{align*}
Besides, the autocatalytic $\tilde{\mathcal{M}}^{(1)}$ satisfies (i) $\mathcal{I}'_1=\{1,2,3,4\}$; (ii) reaction vector balancing with an equilibrium $\tilde{x}^{(1)^*}=(1,1)$ in the positive stoichiometric compatibility class constrained by $\{\tilde{x}_1+\tilde{x}_3=2\}$. Its complexes are
\begin{align*}
&\tilde{v}_{\cdot 1(1)}=\tilde{v}'_{\cdot 5(1)}=(1,0)^\top,
\tilde{v}_{\cdot 5(1)}=\tilde{v}'_{\cdot 1(1)}=(0,1)^\top,
\tilde{v}_{\cdot 2(1)}=(1,1)^\top,
\tilde{v}'_{\cdot 2(1)}=(0,2)^\top,\notag\\
&\tilde{v}_{\cdot 3(1)}=(1,2)^\top,
\tilde{v}'_{\cdot 3(1)}=(0,3)^\top,
\tilde{v}_{\cdot 4(1)}=(1,3)^\top,
\tilde{v}'_{\cdot 4(1)}=(0,4)^\top.
\end{align*}
From \cref{le:CBP-autoLya}, its Lyapunov function PDE has a solution in the form of
$$f(\tilde{x})=3-\tilde{x}_{1}-\tilde{x}_{1}\ln \tilde{x}_1-\frac{1}{2}\tilde{x}_2-\frac{1}{2}\tilde{x}_2
\ln \frac{4}{\tilde{x}_{2}}
+\int^{\tilde{x}_3}_{1}
\ln \frac{12\tilde{x}_{3}}{8+2\tilde{x}_3+\tilde{x}^2_3+\tilde{x}^3_3}d\tilde{x}_3.$$
Then the Hessian matrix of $f(x)$ is computed by
$$\nabla^2f(\tilde{x})
=\emph{diag}\bigg(\tilde{x}^{-1}_1,\frac{1}{2}\tilde{x}^{-1}_2,\frac{8-\tilde{x}^2_3-2\tilde{x}^3_3}{\tilde{x}_{3}(8+2\tilde{x}_3+\tilde{x}^2_3+\tilde{x}^3_3)}\bigg).$$
It is evident to verify that $8-\tilde{x}^2_3-2\tilde{x}^3_3|_{\tilde{x}^*=(1,4,1)}>0$.
Finally, in terms of \cref{thm:CBPauto}, $f(\tilde{x})$ can behave like a Lyapunov function to prove this network locally asymptotically stable at the equilibrium $\tilde{x}^*=(1,4,1)^\top$.
\end{example}

 Inspired by the CBP-$\ell$ts-Autoca MAS, a method similar to dimensionality reduction is proposed in the following corollary, which shows that if a MAS can be decomposed into a CBP MAS and some $1$-dimensional MASs, then the stability of its equilibria can be achieved under certain conditions.

\begin{corollary}
If an MAS with a positive equilibrium $x^*\in\mathbb{R}^n_{>0}$ can be decomposed into a CBP $\mathcal{M}^{(0)}=(\mathcal{S}^{(0)},\mathcal{C}^{(0)},\mathcal{R}^{(0)},\mathcal{K}^{(0)})$ and $\ell$ independent two-species autocatalytic MASs according to species, labeled by $\mathcal{M}^{(p)}=(\mathcal{S}^{(p)},\mathcal{C}^{(p)},\mathcal{R}^{(p)},\mathcal{K}^{(p)}),~p=1,\cdots,\ell$, and moreover, $S_p=\mathcal{S}^{(0)}\bigcap \mathcal{S}^{(p)}$ with $\mathcal{{S}}^{(0)}=\{S_1,\cdots,S_{n_0}\}, \mathcal{{S}}^{(p)}=\{S_p,S_{n_0+p}\}$, $\ell\leq n_0$ and $n_0+\ell=n$, then $x^*$ is locally asymptotically stable with the following conditions to be true
\begin{enumerate}
\item [\emph{(1)}]
for every $S_p$, $\exists v_{\cdot i(0)}\to v'_{\cdot i(0)}\in\mathcal{M}^{(0)}$ such that $v_{pi(0)}=1$ while $v'_{pi(0)}=0$, and $\forall v_{\cdot i(p)}\to v'_{\cdot i(p)}\in \mathcal{M}^{(p)}$ such that $v_{pi(p)},v'_{pi(p)}$ equal to $0$ or $1$;
\item [\emph{(2)}] $\sum_{m\in \mathcal{I}'_p}(2-m){k}^{(p)}_{m,1}
{x}^{* m-1}_{n_0+p}>0$.
\end{enumerate}
\end{corollary}

\begin{proof}
Combing the results in \cref{le:CBP-autoLya} and \cref{thm:CBPauto} about the CBP-$\ell$ts-Autoca MAS,  the conclusion follows immediately.
\end{proof}
\section{Conclusions}\label{sec6}
In this paper, CBP MASs, CBP-$\ell$Sub1 MASs and CBP-$\ell$ts-Autoca MASs are consecutively defined from a complex balanced CRN according to some rules, following which an algorithm is proposed to compute CBP MASs systematically. All of these three classes of networks can be any dimensional, non-weakly reversible and of arbitrary deficiency. Moreover, for CBP MASs/CBP-$\ell$ts-Autoca MASs it has been shown that each positive stoichiometric  compatibility class contains a unique/at most a positive equilibrium. We use the Lyapunov functions PDEs method to successfully catch locally asymptotic stability of these three classes of MASs. The result greatly supports our previous conjecture \cite{Fang2015Lyapunov} that the Lyapunov function PDEs of every stable MAS have a solution capable of acting as a Lyapunov function to render the asymptotic stability.


\section*{Appendix}
\subsection*{A. Proofs of results in \cref{sec4}}

\begin{proof}[\textbf{The proof of \cref{lemmaGao}}]
Observe that every $\tilde{\mathcal{M}}^{(p)}$ for $p=0,\cdots,{\ell}$ contained in the CBP-$\ell$Sub1 $\tilde{\mathcal{M}}$ are mutually independent according to species, we can directly acquire that $f(\tilde{x})$ is a solution of the Lyapunov function PDE \cref{eq:CBP-sub1PDE} as long as combine \cref{thm:sub1} and \cref{CBP-CRNstability}.
\end{proof}

\begin{proof}[\textbf{The proof of \cref{thm:CBP-sub1}}]
Since $h_p(\tilde{x}^{(p)},{u}^{(p)})$ is continuous in $\mathbb{R}^{n_p}_{>0}\times \mathbb{R}_{>0}$ and $\omega^\top_p \frac{\partial}{\partial\tilde{x}^{(p)}} h_p(\tilde{x}^{(p)^*},1)<0$, there must exist some neighbourhoods of $\tilde{x}^{(p)^*}$ for $p=1,\cdots,{\ell}$, denoted by $\mathcal{N}(\tilde{x}^{(p)^*})$, such that for any $\tilde{x}^{(p)}\in \mathcal{N}(\tilde{x}^{(p)^*})$, it holds
$$\omega^\top_p \frac{\partial}{\partial\tilde{x}^{(p)}} h_p(\tilde{x}^{(p)},\tilde{u}^{(p)})<0,$$
where $\tilde{u}^{(p)}$ makes $h_p(\tilde{x}^{(p)},{u}^{(p)})=0$. Thus,  $\forall \mu \in \tilde{\mathscr{S}}$, $\forall \tilde{x}\in
\{\mathbb{R}^{n_0}_{>0}\bigotimes^{\ell}_{p=1} \mathcal{N}(\tilde{x}^{(p)^*})\}\bigcap
\tilde{\mathscr{S}}^+(\tilde{x}^*)$ there is
\begin{align}\label{eq:secderivative1}
\mu^\top\nabla^2f(\tilde{x})\mu=&\mu^{(0)^\top}\text{diag}
\left\{{d_1}/{\tilde{x}^{(0)}_1},\cdots,{d_{n_0}}/{\tilde{x}^{(0)}_{n_0}}\right\}\mu^{(0)}
+\sum^{\ell}_{p=1} \mu^{(p)^\top} \nabla^2 f_p(\tilde{x}^{(p)})\mu^{(p)},
\end{align}
where $f_p(\tilde{x}^{(p)})$ is defined according to \cref{thm:sub1}. Clearly, the first term in the above equation is non-negative. Since $\tilde{u}^{(p)}=\exp\{\omega^\top_p\nabla f_p\}$, we have $\nabla^2f_p \omega_p=\frac{\nabla \tilde{u}^{(p)}}
{\tilde{u}^{(p)}}$. Further, we get
\begin{align}\label{eq:secderivative2}
\sum^{\ell}_{p=1} \mu^{(p)^\top} \nabla^2 f_p(\tilde{x}^{(p)})\mu^{(p)}
&=\sum^{\ell}_{p=1}
\frac{\mu^{(p)^\top}\mu^{(p)}}{\omega^\top_p\omega_p}
\omega^\top_p \nabla^2 f_p(\tilde{x}^{(p)})\omega_p
\notag\\
&=\sum^{\ell}_{p=1}
\frac{\mu^{(p)^\top}\mu^{(p)}}{\omega^\top_p\omega_p}
\omega^\top_p
\frac{\nabla \tilde{u}^{(p)}(\tilde{x}^{(p)})}
{\tilde{u}^{(p)}}
\notag\\
&=
\sum^{\ell}_{p=1} \frac{\mu^{(p)^\top}\mu^{(p)}}{\omega^\top_p\omega_p}
\frac{-\omega^\top_p\frac{\partial}{\partial \tilde{x}^{(p)}} h_p(\tilde{x}^{(p)},\tilde{u}^{(p)})
	/
	\frac{\partial}{\partial {u}^{(p)}} h_p(\tilde{x}^{(p)},\tilde{u}^{(p)})}
{\tilde{u}^{(p)}}
\notag\\
&\geq0.
\end{align}
Note that it's easy to get $\frac{\partial}{\partial {u}^{(p)}} h_p(\tilde{x}^{(p)},\tilde{u}^{(p)})>0$ in penultimate equality just by some simple calculations.
Therefore, we can see $\mu^\top\nabla^2f(\tilde{x})\mu\geq0$ where the equality holds iff $\mu=\mathbbold{0}_n$. Finally utilizing \cref{asymstability}, the result comes quickly, that $\tilde{x}^*$ is locally asymptotically stable.
\end{proof}

\subsection*{B. Proofs of results in \cref{sec5}}

\begin{proof}[\textbf{The proof of \cref{pro:3}}]
Just inserting the $\tilde{x}^{(0)^*}$ and $\tilde{x}^{(p)^*}$ ($p=1, \cdots,{\ell}$) into the dynamics of the CBP-$\ell$ts-Autoca MAS governed by \cref{eq:dynamicCBP-auto}, it is easy to know that $\tilde{x}^*$ is a positive equilibrium point of the considered $\tilde{\mathcal{M}}$ distinctly, and vice versa.
\end{proof}

\begin{proof}[\textbf{The proof of \cref{le:CBP-autoequilibriumunique}}]
	The results can be proved with the aid of the dynamics of the considered $\tilde{\mathcal{M}}$, which is specifically stated as
	\begin{align*}
	\left\{
	\begin{array}{ll}
	\dot {\tilde{x}}_{p}&=-\sum_{m\in \mathcal{I}'_p}{k}^{(p)}_{m,1}\tilde{x}_{p} \tilde{x}^{m-1}_{n_0+p}+{k}^{(p)}_{2} \tilde{x}_{n_0+p}
	+ \sum^{r_0}_{i=1}
	\tilde{k}^{(0)}_i \tilde{x}^{v^{(0)}_{\cdot i}}
	\left(v'^{(0)}_{pi}-v^{(0)}_{pi}\right)\\
	\dot {\tilde{x}}_{j}&=\sum^{r_0}_{i=1}\tilde{k}^{(0)}_i \tilde{x}^{\tilde{v}^{(0)}_{\cdot i}}
	\left(\tilde{v}'^{(0)}_{ji}-\tilde{v}^{(0)}_{ji}\right),~~~j={\ell}+1,\cdots,n_0 \\
	\dot {\tilde{x}}_{n_0+p}&=\sum_{m\in \mathcal{I}'_p}{k}^{(p)}_{m,1} \tilde{x}_{p} \tilde{x}^{m-1}_{n_0+p}-{k}^{(p)}_{2} \tilde{x}_{n_0+p},~~~p=1,\cdots,{\ell}.
	\end{array}
	\right.
	\end{align*}
	It demonstrates that the whole $\tilde{\mathcal{M}}$ is balanced if and only if the involved CBP $\tilde{\mathcal{M}}^{(0)}$ and autocatalytic $\tilde{\mathcal{M}}^{(p)}$s are both balanced. Since \cref{pro:CBP} reveals that the $\tilde{\mathcal{M}}^{(0)}$ possesses a unique positive equilibrium in every positive stoichiometric class, then the number of equilibrium in each positive stoichiometric class induced by the CBP-$\ell$ts-Autoca $\tilde{\mathcal{M}}$ can be decided by the remaining  $\tilde{\mathcal{M}}^{(p)}$s. Suppose the equilibrium $\tilde{x}^{(0)^*}\in \mathbb{R}^{n_0}_{>0}$ of $\tilde{\mathcal{M}}^{(0)}$ is given, then the remaining dynamic equations turn to be
	$$\dot {\tilde{x}}_{n_0+p}=\sum_{m\in \mathcal{I}'_p}{k}^{(p)}_{m,1} \tilde{x}^*_{p} \tilde{x}^{m-1}_{n_0+p}-{k}^{(p)}_{2} \tilde{x}_{n_0+p}.$$ 	
	Apparently, when $\mathcal{I}'_p=\{1\}$, $\dot x_{n_0+p} =0$ has a unique positive solution while $\mathcal{I}'_p=\{1,2\}$ it has precisely one positive solution only if $k^{(p)}_2-\tilde{x}^*_p k^{(p)}_{2,1}>0$. Thus $\forall \{p\}^{\ell}_{p=1}$, when $\tau_p\leq2$, there is at most one equilibrium in $\tilde{\mathscr{S}}^+(\tilde{x}_0)$ for every initial state $\tilde{x}_0 \in \mathbb{R}^n_{>0}$.
	
Next, let $W_{\mathcal{I}'_p}\subseteq \{1,\cdots,{\ell}\}$ represent an index set that satisfies $p\in W_{\mathcal{I}'_p}$ if $\mathcal{I}'_p$ includes an element greater than 2, i.e. $\exists~ m>2$. Then for the $\mathcal{M}^{(p)}$ with $p\in W_{\mathcal{I}'_p}$, there will be at most two positive intersection points when the polynomial $\sum_{m\in \mathcal{I}'_p}k^{(p)}_{m,1} \tilde{x}^*_{p} \tilde{x}^{m-1}_{n_0+p}$ intersects with $k^{(p)}_{2} \tilde{x}_{n_0+p}$ at the plane $\tilde{x}_{n_0+p}>0$. Since such $\tilde{\mathcal{M}}^{(p)}$ is mass-conserved, we know that when $\tilde{x}^*_l$ is fixed there is at most one positive equilibrium in its positive stoichiometric compatibility class. Therefore, it indicates that for an $\tilde{\mathcal{M}}$ there is at most an equilibrium in $ \tilde{\mathscr{S}}^+(\tilde{x}_0)\bigcap\mathscr{M}$ for any initial condition $\tilde{x}_0\in\mathbb{R}^n_{>0}\bigcap( \bigcup_ {p\in W_{\mathcal{I}'_p}}\mathscr{M}_p)$, where $\mathscr{M}_p=\{\tilde{x}_p+\tilde{x}_{n_0+p}=M_p, M_p>0 \}$ represents the conservation law that $\tilde{\mathcal{M}}^{(p)}$ follows.
	\end{proof}

\begin{proof}[\textbf{The proof of \cref{le:CBP-autoLya}}]
First of all, the corresponding Lyapunov function PDE for the CBP-$\ell$ts-Autoca $\tilde{\mathcal{M}}$ can be written as
\begin{align}\label{eq:CBP-autoPDE}
	\sum^{\ell}_{p=0}
	\sum^{r_p}_{i=1}
	\left(
	\tilde{k}^{(p)}_{i} \tilde{x}^{(p)\tilde{v}_{\cdot i (p)}}-
	\tilde{k}^{(p)}_{i} \tilde{x}^{(p)\tilde{v}_{\cdot i (p)}}\exp \left\{(\tilde{v}'_{\cdot i(p)}-\tilde{v}_{\cdot i(p)})^\top \frac{\partial{f(\tilde{x})}}{\partial \tilde{x}^{(p)}}\right \}\right)=0.
	\end{align}
	Then taking
	$$\nabla f(\tilde{x})=\bigg((d_i\ln \frac{\tilde{x}_i}{\tilde{x}^*_i})^{n_0}_{i=1},
	(\ln \frac{{k}^{(p)}_{2}\tilde{x}_{n_0+p}}{\sum_{m\in \mathcal{I}'_p}{k}^{(p)}_{m,1} \tilde{x}^*_{p} \tilde{x}^{m-1}_{n_0+p}})^{\ell}_{p=1}\bigg)$$
	into \cref{eq:CBP-autoPDE}, we derive that
	\begin{align}\label{eq:CBPPDE}
	\sum^{r_0}_{i=1}
	\tilde{k}^{(0)}_{i} \tilde{x}^{(0)\tilde{v}_{\cdot i (0)}}
	\bigg(
	1-
	\exp \left\{(\tilde{v}'_{\cdot i(0)}-\tilde{v}_{\cdot i(0)})^\top \ln \frac{\tilde{x}^{(0)}}{\tilde{x}^{(0)^*}}\right \}
	\bigg)=0,
	\end{align}
	where the second equality holds on account of \cref{CBP-CRNstability}.
	
	Since $d_p=1$ for $p=1,\cdots,{\ell}$, we get $\zeta=\frac{\partial{f(\tilde{x})}}{\partial{\tilde{x}}^{(p)}}=\bigg(\ln\frac{\tilde{x}_{p}}{\tilde{x}^*_{p}},
	~\ln \frac{{k}^{(p)}_{2}\tilde{x}_{n_0+p}}{\sum_{m\in \mathcal{I}'_p}
		{k}^{(p)}_{m,1} \tilde{x}^*_{p} {x}^{m-1}_{n_0+p}}\bigg)$.
For any $p\in \{1,\cdots,{\ell}\}$, we have
\begin{align}\label{eq:autoPDE}
&\sum^{r_p}_{i=1}
    \bigg(
    \tilde{k}^{(p)}_{i} \tilde{x}^{(p)\tilde{v}_{\cdot i (p)}}-
	\tilde{k}^{(p)}_{i} \tilde{x}^{(p)\tilde{v}_{\cdot i (p)}}\exp \left\{(\tilde{v}'_{\cdot i(p)}-\tilde{v}_{\cdot i(p)})^\top \frac{\partial{f(\tilde{x})}}{\partial \tilde{x}^{(p)}}\right \} \bigg)\\
=&
\sum_{m\in \mathcal{I}'_p}k^{(p)}_{m,1} \tilde{x}_{p}\tilde{x}^{m-1}_{n_0+p}
\bigg(
1-\exp \left \{ (-1,1)^\top
\zeta
\right \}
\bigg)
+k^{(p)}_{2}\tilde{x}_{n_0+p}
\bigg(1-\exp \left\{
(1,-1)^\top
\zeta
\right\}
\bigg)\notag\\
=&
\sum_{m\in \mathcal{I}'_p}k^{(p)}_{m,1} \tilde{x}_{p}\tilde{x}^{m-1}_{n_0+p}
\bigg(
1-
\frac{k^{(p)}_{2}\tilde{x}_{n_0+p}}{\sum_{m\in \mathcal{I}'_p}
	k^{(p)}_{m,1} \tilde{x}_{p} \tilde{x}^{m-1}_{n_0+p}}
\bigg)
+
k^{(p)}_{2}\tilde{x}_{n_0+p}
\bigg(1-\frac{\sum_{m\in \mathcal{I}'_p}
	k^{(p)}_{m,1} \tilde{x}_{p} \tilde{x}^{m-1}_{n_0+p}}{k^{(p)}_{2}\tilde{x}_{n_0+p}}
\bigg)\notag\\
=&0 \notag,
\end{align}	
	
Therefore, by summing \cref{eq:autoPDE} from $p=1$ to ${\ell}$ and \cref{eq:CBPPDE} we get \cref{eq:CBP-autoPDE}.
\end{proof}

\begin{proof}[\textbf{The proof of \cref{thm:CBPauto}}]
	Firstly, we compute the second derivative of $f(\tilde{x})$ as
	\begin{equation}\label{eq:Hessian}
	\begin{array}{ll}
	\nabla^2 f(\tilde{x})=
	\left(
	\begin{array}{cc}
	\text{diag}(d_i \tilde{x}^{-1}_{i})^{n_0}_{i=1} & \mathbbold{0}_{{n_0}\times {\ell}} \\
	\mathbbold{0}_{{\ell}\times{n_0} }& \text{diag}\bigg(\frac{\sum_{m\in \mathcal{I}'_p}(2-m){k}^{(p)}_{m,1}
		\tilde{x}^{m-1}_{n_0+p}}{\sum_{m\in \mathcal{I}'_p}{k}^{(p)}_{m,1}\tilde{x}^m_{n_0+p}}\bigg)^{\ell}_{p=1}
	\end{array}
	\right).
	\end{array}
	\end{equation}
	Clearly, $\forall \{p\}^{\ell}_{p=1}$, when $\tau_p\leq2$, which means $m\leq2$, thus $f{(\tilde{x})}$ is strictly convex for $\nabla^2 f(\tilde{x}) > 0 $ in $\tilde{\mathscr{S}}^+(\tilde{x}_0)$ with respect to the initial value $\tilde{x}_0$ near $\tilde{x}^*$. Associated with \cref{asymstability}, $\tilde{x}^*$ is locally asymptotically stable.
	
	Further, we continue to prove the second result. The continuity of the function $\sum_{m\in \mathcal{I}'_p}(2-m)k^{(p)}_{m,1}
	\tilde{x}^{m-1}_{n_0+p}$ with respect to $\tilde{x}_{n_0+p}$ implies that, there exists a neighborhood of $\tilde{x}^*_{n_0+p}$ for $p=1,\cdots,\ell$, denoted by $\mathcal{N}(\tilde{x}^*_{n_0+p})$, such that $\forall \tilde{x}_{n_0+p} \in\mathcal{N}(\tilde{x}^*_{n_0+p})$, it holds
	$$\sum_{m\in \mathcal{I}'_p}(2-m)k^{(p)}_{m,1}
	\tilde{x}^{m-1}_{n_0+p}>0.$$
	Therefore,
	$\forall \tilde{x} \in \bigg(\mathbb{R}^{n_0}_{>0}\bigotimes^{\ell}_{p=1}\mathcal{N}(\tilde{x}^*_{n_0+p})\bigg)\bigcap \tilde{\mathscr{S}}^+(\tilde{x}^*)\bigcap( \bigcup_ {p\in W_{\mathcal{I}'_p}}\mathscr{M}_p)$, where $\mathscr{M}_p,W_{\mathcal{I}'_p}$ share the same meanings as the ones given in the proof of \cref{le:CBP-autoequilibriumunique}, we have $\nabla^2 f(\tilde{x}) >0 $. In the end, \cref{asymstability} tells us that $\tilde{x}^*$ is locally asymptotically stable.
\end{proof}

\section*{Acknowledgments}
We would like to acknowledge Arjan van der Schaft for his valuable comments.

\bibliographystyle{siamplain}

\end{document}

%% file: ex_shared.tex
\usepackage{lipsum}
\usepackage{amsfonts}
\usepackage{graphicx}
\usepackage{epstopdf}
\usepackage{algorithmic}
\ifpdf
  \DeclareGraphicsExtensions{.eps,.pdf,.png,.jpg}
\else
  \DeclareGraphicsExtensions{.eps}
\fi

\newsiamremark{remark}{Remark}
\newsiamremark{hypothesis}{Hypothesis}
\crefname{hypothesis}{Hypothesis}{Hypotheses}
\newsiamthm{claim}{Claim}

\title{Lyapunov Function PDEs Method to the Stability of Some Chemical Reaction Networks \thanks{This work will be partly presented in 21st IFAC World Congress in Berlin, Germany, July 12-17, 2020. Submitted to the editors. 
\funding{This work was funded by the National Nature Science Foundation of China under Grant
No. 11671418, and the Zhejiang Provincial Natural Science Foundation of China under Grant No.
LZ20A010002.
}}}

\author{Yafei Lu\thanks{School of Mathematical Sciences, Zhejiang University (\email{11535029@zju.edu.cn (Y. Lu), gaochou@zju.edu.cn (C. Gao, Correspondence)).}}
\and Chuanhou Gao\footnotemark[2]}

\usepackage{amsopn}


%% file: mainnew.bbl
\begin{thebibliography}{10}
	
	\bibitem{Alradhawi2016New}
	{\sc M.~A. Alradhawi and D.~Angeli}, {\em New approach to the stability of
		chemical reaction networks: Piecewise linear in rates lyapunov functions},
	IEEE T. Automat. Contr., 61 (2016), pp.~76--89.
	
	\bibitem{Anderson2015Lyapunov}
	{\sc D.~F. Anderson, C.~Gheorghe, G.~Manoj, and W.~Carsten}, {\em Lyapunov
		functions, stationary distributions, and non-equilibrium potential for
		reaction networks}, Bull. Math. Biol., 77 (2015), pp.~1744--1767.
	
	\bibitem{Angeli2009A}
	{\sc D.~Angeli}, {\em A tutorial on chemical reaction network dynamics}, Eur.
	J. Control., 15 (2009), pp.~398--406.
	
	\bibitem{Angeli2007A}
	{\sc D.~Angeli, P.~D. Leenheer, and E.~D. Sontag}, {\em A petri net approach to
		the study of persistence in chemical reaction networks}, Math. Biosci., 210
	(2007), pp.~598--618.
	
	\bibitem{Cappelletti2018Graphically}
	{\sc D.~Cappelletti and B.~Joshi}, {\em Graphically balanced equilibria and
		stationary measures of reaction networks.}, SIAM J. Appl. Dyn. Syst., 17
	(2018), pp.~2246--2175.
	
	\bibitem{Craciun2006Multiple}
	{\sc G.~Craciun and M.~Feinberg}, {\em Multiple equilibria in complex chemical
		reaction networks: Ii. the species-reaction graph}, SIAM J. Appl. Math., 66
	(2006), pp.~1321--1338.
	
	\bibitem{Craciun2013Persistence}
	{\sc G.~Craciun, F.~Nazarov, and C.~Pantea}, {\em Persistence and permanence of
		mass-action and power-law dynamical systems}, SIAM J. Appl. Math., 73 (2013),
	pp.~305--329.
	
	\bibitem{Fang2015Lyapunov}
	{\sc Z.~Fang and C.~Gao}, {\em Lyapunov function partial differential equations
		for chemical reaction networks: Some special cases}, SIAM J. Appl. Dyn.
	Syst., 18 (2019), pp.~1163--1199.
	
	\bibitem{Feinberg1972Complex}
	{\sc M.~Feinberg}, {\em Complex balancing in general kinetic systems}, Arch.
	Ration. Mech. Anal., 49 (1972), pp.~187--194.
	
	\bibitem{Finberg1987Chemical}
	{\sc M.~Feinberg}, {\em Chemical reaction network structure and the stability
		of complex isothermal reactors-i. the deficiency zero and deficiency one
		theorems}, Chem. Eng. Sci., 42 (1987), pp.~2229--2268.
	
	\bibitem{Finberg1988Chemical}
	{\sc M.~Feinberg}, {\em Chemical reaction network structure and the stability
		of complex isothermal reactors-ii.multiple steady states for networks of
		deficiency one}, Chem. Eng. Sci., 43 (1988), pp.~1--25.
	
	\bibitem{Necessary1989}
	{\sc M.~Feinberg}, {\em Necessary and sufficient conditions for detailed
		balancing in mass action systems of arbitrary complexity}, Chem. Eng. Sci.,
	44 (1989), pp.~1819--1827.
	
	\bibitem{Feinberg1995The}
	{\sc M.~Feinberg}, {\em The existence and uniqueness of steady states for a
		class of chemical reaction networks}, Arch. Ration. Mech. Anal., 132 (1995),
	pp.~311--370.
	
	\bibitem{Gopalkrishnan2011}
	{\sc M.~Gopalkrishnan}, {\em Catalysis in reaction networks.}, Bull. Math.
	Biol., 73 (2011), pp.~2962--2982.
	
	\bibitem{GrimbsSpatiotemporal}
	{\sc S.~Grimbs, A.~Arnold, A.~Koseska, J.~Kurths, J.~Selbig, and Z.~Nikoloski},
	{\em Spatiotemporal dynamics of the calvin cycle: Multistationarity and
		symmetry breaking instabilities}, Biosystems, 103 (2011), pp.~212--223.
	
	\bibitem{hoessly2019stationary}
	{\sc L.~Hoessly and C.~Mazza}, {\em Stationary distributions and condensation
		in autocatalytic reaction networks}, SAIM J. Appl. Math., 79 (2019),
	pp.~1173--1196.
	
	\bibitem{Hordijk2004}
	{\sc W.~Hordijk and M.~Steel}, {\em Detecting autocatalytic, self-sustaining
		sets in chemical reaction systems.}, J. Theoret. Biol., 227 (2004),
	pp.~451--461.
	
	\bibitem{HornJackson1972General}
	{\sc F.~Horn and R.~Jackson}, {\em General mass action kinetics}, Arch. Ration.
	Mech. Anal., 47 (1972), pp.~81--116.
	
	\bibitem{Johnston2012Dynamical}
	{\sc M.~D. Johnston, S.~David, and S.~Gabor}, {\em Dynamical equivalence and
		linear conjugacy of chemical reaction networks: New results and methods},
	MATCH Commun. Math. Comput. Chem., 68 (2012), pp.~443--468.
	
	\bibitem{Johnston2011Linear}
	{\sc M.~D. Johnston and D.~Siegel}, {\em Linear conjugacy of chemical reaction
		networks}, J. Math. Chem., 49 (2011), pp.~1263--1282.
	
	\bibitem{Kauffman1995At}
	{\sc S.~A. Kauffman}, {\em At home in the universe: The search for laws of
		self-organization and complexity}, Leonardo, 29 (1995).
	
	\bibitem{Ke2019Complex}
	{\sc M.~Ke, Z.~Fang, and C.~Gao}, {\em Complex balancing reconstructed to the
		asymptotic stability of mass-action chemical reaction networks with
		conservation laws}, SIAM J. Appl. Math., 79 (2019), pp.~55--74.
	
	\bibitem{NenMotif}
	{\sc S.~Nen, S.~Yuki, and K.~Kunihiko}, {\em Motif analysis for small-number
		effects in chemical reaction dynamics}, J. Chem. Phys., 145 (2016),
	pp.~094111--.
	
	\bibitem{Pantea2012On}
	{\sc C.~Pantea}, {\em On the persistence and global stability of mass-action
		systems.}, SIAM J. Math. Anal., 44 (2012), pp.~1636--1673.
	
	\bibitem{rao2013graph}
	{\sc S.~Rao, A.~van~der Schaft, and B.~Jayawardhana}, {\em A graph-theoretical
		approach for the analysis and model reduction of complex-balanced chemical
		reaction networks}, J. Math. Chem., 51 (2013), pp.~2401--2422.
	
	\bibitem{Siegel2000Global}
	{\sc D.~Siegel and D.~Maclean}, {\em Global stability of complex balanced
		mechanisms}, J. Math. Chem., 27 (2000), pp.~89--110.
	
	\bibitem{Sontag2012Structure}
	{\sc E.~D. Sontag}, {\em Structure and stability of certain chemical networks
		and applications to the kinetic proofreading model of t-cell receptor signal
		transduction}, IEEE T. Automat. Contr., 46 (2012), pp.~1028--1047.
	
	\bibitem{Szederke2011Finding}
	{\sc G.~Szederkenyi and K.~M. Hangos}, {\em Finding complex balanced and
		detailed balanced realizations of chemical reaction networks}, J. Math.
	Chem., 49 (2011), pp.~1163--1179.
	
	\bibitem{Wu2020A}
	{\sc S.~Wu, Y.~Lu, and C.~Gao}, {\em Lyapunov function partial differential
		equations for stability analysis of a class of chemical reaction networks.},
	in 21st IFAC World Congress in Berlin, to appear, 2020.
	
\end{thebibliography}
